\documentclass[11pt]{article}

\usepackage[utf8]{inputenc}
\usepackage[T1]{fontenc}

\usepackage{amsmath,amsfonts,amssymb, amsthm}
\usepackage{hyperref}
\usepackage{color}
\usepackage{float}
\usepackage{bm}
\usepackage{graphicx}
\usepackage{breakurl}
\usepackage{xifthen}
\usepackage{stmaryrd}
\usepackage{amsopn}

\usepackage{chngcntr}

\if@twoside
\else 
\fi

\theoremstyle{plain} 
\newtheorem{theorem}{Theorem}[section]
\newtheorem{lemma}[theorem]{Lemma}
\newtheorem{corollary}[theorem]{Corollary}

\theoremstyle{definition}
\newtheorem{definition}[theorem]{Definition}

\theoremstyle{remark}
\newtheorem{remark}[theorem]{Remark}

\newcommand{\primal}{w}
\newcommand{\primalt}{v}

\newcommand{\aux}{\bm{M}}
\newcommand{\auxt}{\bm{N}}

\newcommand{\stiffnessC}{\mathcal{C}}

\newcommand{\vp}{p}
\newcommand{\vpt}{q}
\newcommand{\vphi}{\phi}
\newcommand{\vphit}{\psi}
\newcommand{\vphitt}{\xi}

\newcommand{\vphiRot}{\phi^\perp}
\newcommand{\vphitRot}{\psi^\perp}

\newcommand{\V}{{\bm{V}}}
\newcommand{\Q}{{Q}}
\newcommand{\Vprimal}{W}

\newcommand{\VRegDecomp}{V}
\newcommand{\SpPsi}{\Psi}

\newcommand{\Ltwob}{{\bm{L}^2(\Omega)}}
\newcommand{\LtwobSym}{\bm{L}^2(\Omega)_{\mathrm{sym}}}

\newcommand{\Htwob}[1][]{
\ifthenelse{\isempty{#1}}
{\bm{H}^2(\Omega)}
{\bm{H}^2_0(\Omega)}
}

\newcommand{\HRot}{\bm{H}(\brot, \Omega)}
\newcommand{\Hdiv}{\bm{H}(\Opdiv, \Omega)}
\newcommand{\Opdiv}{\bdiv}

\newcommand{\Hdivdiv}[1]{\bm{H}(\Opdivdiv, \Omega; #1)_{\mathrm{sym}}}
\newcommand{\Opdivdiv}{\div\bdiv}
\newcommand{\normHdivdiv}[2]{\|#1\|_{\Opdivdiv; #2}}

\newcommand{\FE}{\mathcal{S}}

\newcommand{\GammaVertices}[1][]{
\ifthenelse{\isempty{#1}}
{\mathcal{V}_\Gamma}
{\mathcal{V}_{\Gamma,#1}}
}
\newcommand{\GammaEdges}[1][]{
\ifthenelse{\isempty{#1}}
{\mathcal{E}_\Gamma}
{\mathcal{E}_{\Gamma,#1}}
}

\newcommand{\HtwoNV}[1][\Omega]{H^2_{N,\mathcal{V}}(#1)}

\newcommand{\decomposition}{\mathcal{T}_h}
\newcommand{\vertices}[1][]{
\ifthenelse{\isempty{#1}}
{\mathcal{V}_h}
{\mathcal{V}_{h,#1}}
}
\newcommand{\edges}[1][]{
\ifthenelse{\isempty{#1}}
{\mathcal{E}_h}
{\mathcal{E}_{h,#1}}
}

\newcommand{\vertexJump}[1]{\llbracket #1 \rrbracket}

\newcommand{\grad}{\nabla}
\newcommand{\curl}{\operatorname{curl}}
\renewcommand{\div}{\operatorname{div}}
\newcommand{\rot}{\operatorname{rot}}

\newcommand{\bdiv}{\operatorname{Div}}
\newcommand{\bCurl}{\operatorname{Curl}}
\newcommand{\brot}{\operatorname{Rot}}

\newcommand{\symbCurl}{\operatorname{symCurl}}
\newcommand{\symbGrad}{\varepsilon}

\renewcommand{\ker}{\mathrm{Ker}} 

\newcommand{\dom}{D}

\newcommand{\Id}{\bm{I}}

\newcommand{\matrixTrace}{\mathrm{tr}}

\newcommand{\extPsi}[1]{\vphit[#1]}
\newcommand{\extPsiGamma}[1]{\vphit_\Gamma[#1]}
\newcommand{\filter}[1]{\operatorname{P}{#1}}

\DeclareMathOperator{\dotprod}{\cdot}
\newcommand{\pp}[2]{\partial_{#2} #1}

\newcommand{\red}[1]{{\color{black}{#1}}}

\newcommand{\sectref}[1]{Section~\ref{#1}}
\newcommand{\appxref}[1]{Appendix}

\newcommand{\thmref}[1]{Theorem~\ref{#1}}
\newcommand{\lemref}[1]{Lemma~\ref{#1}}

\newcommand{\defref}[1]{Definition~\ref{#1}}
\newcommand{\remref}[1]{Remark~\ref{#1}}

\title{A decomposition result for Kirchhoff plate bending problems and a new discretization approach}
\author{
	Katharina Rafetseder\footnote{Institute of Computational Mathematics, Johannes Kepler University Linz, 4040 Linz, Austria (rafetseder@numa.uni-linz.ac.at, zulehner@numa.uni-linz.ac.at).} \, and
	Walter Zulehner\footnotemark[1]
}

\begin{document}

\maketitle

\begin{abstract}
A new approach is introduced for deriving a mixed variational formulation for Kirchhoff plate bending problems with mixed boundary conditions 
involving clamped, simply supported, and free boundary parts. 
Based on a regular decomposition of an appropriate nonstandard Sobolev space for the bending moments, 
the fourth-order problem can be equivalently written as a system of three (consecutively to solve) second-order problems in standard Sobolev spaces. 
This leads to new discretization methods, which are flexible in the sense, that any existing and well-working discretization method and solution strategy for standard second-order problems can be used as a modular building block of the new method.

Similar results for the first biharmonic problem have been obtained in our previous work [W. Krendl, K. Rafetseder and W. Zulehner, A decomposition result for biharmonic problems and the Hellan-Herrmann-Johnson method, ETNA, 2016]. The extension to more general boundary conditions encounters several difficulties including the construction of an appropriate nonstandard Sobolev space, the verification of Brezzi's conditions, and the adaptation of the regular decomposition.


\end{abstract}


\section{Introduction}
\label{sec:introduction}
We consider the Kirchhoff plate bending problem: For a given load $f$, find the deflection $\primal$ such that 
\begin{equation}\label{eq:class_equation}
	\div\bdiv \big(\stiffnessC  \grad^2 \primal\big) =  f \quad \text{in } \Omega
\end{equation}
with appropriate boundary conditions. Here $\Omega$ is a bounded domain in $\mathbb{R}^2$ with a polygonal Lipschitz boundary $\Gamma$, $\div$ denotes the standard divergence of a vector-valued function, $\bdiv$ the row-wise divergence of a matrix-valued function, $\grad^2$ the Hessian and $\stiffnessC$ the material tensor. Note that \eqref{eq:class_equation} reduces to the biharmonic equation, if $\stiffnessC$ is the identity.
In this paper we focus on mixed methods for the original unknown $\primal$ and the bending moments $\aux = -\stiffnessC \grad^2 \primal$ as additional unknowns, which are often quantities of interest on their own.

In our previous work \cite{krendl_rafetseder_zulehner_2016} the first biharmonic boundary value problem is studied, which corresponds to the situation of a purely clamped plate. For this model problem a new mixed variational formulation is derived, which satisfies Brezzi's conditions and is equivalent to the original problem. However, these important properties come at the cost of an appropriate nonstandard Sobolev space $\V$ for $\aux$. Based on a regular decomposition
of $\V$, the fourth-order problem can be rewritten as a sequence of three (consecutively to solve) second-order elliptic problems in standard Sobolev spaces. This leads to a new interpretation of known mixed finite element methods as well as to the construction of new discretization methods, see \cite{krendl_rafetseder_zulehner_2016} for details.
This approach fits into an abstract framework recently presented in \cite{chen_huang_2016} for the decomposition of higher-order problems. However, more general boundary conditions (beyond a purely clamped plate) are not considered there as well. 

The aim of this paper is to extend the ideas of \cite{krendl_rafetseder_zulehner_2016} to the more general situation of a Kirchhoff plate with mixed boundary conditions involving clamped, simply supported, and free boundary parts. This is by far not straight-forward.

The first difficulty arises in the derivation of the new mixed formulation. Contrary to clamped boundary parts, appropriate boundary conditions for $\aux$ have to be incorporated into the definition of the nonstandard Sobolev space $\V$ for simply supported and free boundary parts.  In this paper we do this indirectly using the framework of (possibly unbounded) densely defined operators in Hilbert spaces. 
This approach avoids the direct use of trace operators in nonstandard Sobolev spaces, which would be technically rather involved.
A second difficulty is the verification of Brezzi's conditions. In \cite{krendl_rafetseder_zulehner_2016} the main ingredient for the proof of an inf-sup condition is the property that matrix-valued functions of the form $\primalt \Id$ are contained in $\V$, where $\Id$ denotes the identity matrix and $\primalt$ satisfies homogeneous Dirichlet boundary conditions induced by the boundary conditions for $\primal$. This inclusion is no longer true for problems with free boundary parts. So, a new technique is required for proving the inf-sup condition.
A third difficulty arises in the regular decomposition for a similar reason. The first component of the decomposition in \cite{krendl_rafetseder_zulehner_2016} is of the form $\vp \Id$ and is not contained in $\V$ for problems with free boundary parts. So, a new approach is required to pursue the decomposition for problems with free boundary parts. 
It is shown in this paper how to overcome all these difficulties and how to achieve again a decomposition of the fourth-order problem into three (consecutively to solve) second-order elliptic problems in standard Sobolev spaces.

So far in literature, mixed methods for \eqref{eq:class_equation} in $\primal$ and $\aux$ have been formulated as linear operator equations in function spaces for which either the associated linear operator is not an isomorphism or the involved norms contain a mesh-dependent variant of the $H^2$-norm for $\primal$, see, e.g., \cite{brezzi_raviart_1977,falk_osborn_1980,babuska_osborn_pitkaranta_1980,blum_rannacher_1990,boffi_brezzi_fortin_2013}.
This lack of easy-to-access knowledge on the mapping properties of the involved operators on the continuous level makes it hard to design efficient preconditioners on the discrete level.
Our new mixed formulation satisfies Brezzi's conditions. Therefore, the associated linear operator is an isomorphism. Additionally, the operator is of triangular structure. This naturally leads to the construction of efficient solvers. Moreover, the new mixed formulation is equivalent to the original problem without additional convexity assumptions on $ \Omega$, while most of the papers in literature (except for \cite{blum_rannacher_1990}) require $\Omega$ to be convex.  

For alternative discretization methods such as conforming, non-conforming, and discontinuous Galerkin methods for the primal formulation as well as alternative mixed methods we refer to the short discussion in \cite{krendl_rafetseder_zulehner_2016} and the references cited there.
Our approach leads to new discretization methods, which are flexible in the sense, that any existing and well-working discretization method and solution strategy for second-order problems can be used as a modular building block of the new method.
\red{One option would be to choose standard $C^0$ finite elements for each of the three second-order elliptic problems (for two scalar fields and one vector field) resulting in approximate solutions to $\primal$ and $\aux$. 
In \cite{beirao_niiranen_stenberg_2007,beirao_niiranen_stenberg_2008} a different method was proposed that also uses only standard $C^0$ finite element spaces for second-order problems for a formulation in the kinematic variables $\primal$ and $\grad \primal$. For reaching approximate solutions of comparable accuracy the method in \cite{beirao_niiranen_stenberg_2007,beirao_niiranen_stenberg_2008} requires the approximation of one scalar field less than the approach presented here. 
However, the linear system resulting from the method in \cite{beirao_niiranen_stenberg_2007,beirao_niiranen_stenberg_2008} is a coupled system of all degrees of freedoms of one scalar and one vector field, while  the method presented here requires to solve linear systems for the degrees of freedom separately for each of the two scalar and the vector field.
This reduces the computational costs for direct solvers. For the use of iterative solvers efficent methods for standard second-order problems like multigrid methods can be directly used for each of the three linear systems. 
Preconditioning is not addressed in \cite{beirao_niiranen_stenberg_2007,beirao_niiranen_stenberg_2008}. 
So we feel that our method is competitive with respect to the overall computational efficiency.
}

The paper is organized as follows. In \sectref{sec:kirchhoff_love_plate} the Kirchhoff plate bending problem is introduced. \sectref{sec:new_mixed_formulation} contains a new mixed formulation, for which well-posedness and equivalence to the original problem is shown. 
A regular decomposition of the nonstandard Sobolev space for $\aux$ is derived in \sectref{sec:regular_decomposition} and the resulting decoupled formulation is presented. The decoupled formulation leads in a natural way to the construction of a new discretization method, which is introduced in \sectref{sec:conforming_method}, and for which a priori error estimates are derived. The paper closes with numerical experiments in \sectref{sec:numerical_experiments}. 


\section{The Kirchhoff plate bending problem}
\label{sec:kirchhoff_love_plate}
We consider the Kirchhoff plate bending problem of a linearly elastic plate where the undeformed mid-surface is described by a domain $\Omega \subset \mathbb{R}^2$ with a polygonal Lipschitz boundary $\Gamma$. 
In what follows, let the boundary $\Gamma$ be written in the form
\begin{equation*}
	\Gamma = \GammaVertices \cup \GammaEdges \quad \text{with} \ \GammaEdges = \bigcup_{k=1}^K E_k,
\end{equation*}
where $E_k$, $k = 1,2,\ldots,K$, are the edges of $\Gamma$, considered as open line segments and $\GammaVertices$ denotes the set of corner points in $\Gamma$. Furthermore, $n = (n_1,n_2)^T$ and $t= (-n_2,n_1)^T$ represent the unit outer normal vector and the unit counterclockwise tangent vector to $\Gamma$, respectively.

The plate is considered to be clamped on a part $\Gamma_c \subset \Gamma$, simply supported on $\Gamma_s \subset \Gamma$, free on $\Gamma_f \subset \Gamma$  with $\Gamma = \Gamma_c \cup \Gamma_s \cup \Gamma_f$.  We assume that each edge $E \in \GammaEdges$ is contained in exactly one of the sets $\Gamma_c$, $\Gamma_s$, $\Gamma_f$, and the edges are maximal in the sense that two edges with the same boundary condition do not meet at an angle of $\pi$. 
Recall the definition of the bending moments $\aux$ by the Hessian of the deflection $\primal$
\begin{equation} \label{eq:aux}
	\aux = -\stiffnessC \grad^2 \primal,
\end{equation}
where $\stiffnessC$ is the fourth-order material tensor. 
The tensor $\stiffnessC$ is assumed to be symmetric and positive definite on symmetric matrices, $\lambda_{min}(\stiffnessC)$ and \red{$\lambda_{max}(\stiffnessC)$} denote 
the minimal and maximal eigenvalue of $\stiffnessC$, respectively.
For example, for isotropic materials with Poisson ratio $\nu$, the material tensor $\stiffnessC$ is given by $\stiffnessC \auxt = D \bigl( (1-\nu)\auxt + \nu \ \matrixTrace(\auxt) \Id \bigr)$
for matrices $\auxt$, where $D>0$ depends on material constants, $\Id$ is the identity matrix and $\matrixTrace$ is the trace operator for matrices (cf. \cite{reddy_2007}).
We introduce the following notations
\begin{equation*}
	\aux_{nn} = \aux n \dotprod n, \qquad \aux_{nt} = \aux n \dotprod t
\end{equation*}
for the normal-normal component and the normal-tangential component of $\aux$, where the symbol $\dotprod$ denotes the Euclidean inner product.
The classical Kirchhoff plate bending problem reads as follows (cf. \cite{reddy_2007}): For given load $f$, find a deflection $\primal$ such that
\red{
\begin{equation}\label{eq:class_formulation}
	-\div\bdiv \aux =  f \quad  \text{in } \Omega, \quad \text{with} \ \aux = -\stiffnessC \grad^2 \primal
\end{equation}
}
and the boundary conditions
\begin{equation*} 
\begin{alignedat}{4} 
	&\primal = 0, \quad &&\pp{\primal}{n}=0 & \quad \text{on} \ \Gamma_c, \\
	&\primal = 0, \quad &&\aux_{nn} = 0 & \quad \text{on} \ \Gamma_s, \\
	&\aux_{nn} = 0, \quad &&\pp{\aux_{nt}}{t} + \bdiv \aux \dotprod n = 0 & \quad \text{on} \ \Gamma_f,
\end{alignedat}
\end{equation*}
where $\pp{}{t}$ denotes the tangential derivative, and the corner conditions
\begin{equation}\label{eq:corner_conditions}
	\vertexJump{\aux_{nt}}_x = \aux_{n_1 t_1}(x) - \aux_{n_2 t_2}(x) = 0  \quad \text{for all} \ x \in \GammaVertices[f],
\end{equation}
where $\GammaVertices[f]$ denotes the set of corner points whose two adjacent edges (with corresponding normal and tangent vectors $n_1$, $t_1$ and $n_2$, $t_2$)  belong to $\Gamma_f$.
\begin{remark}
	There is a fourth type of boundary condition given by
	\begin{equation*}
		\pp{\primal}{n}=0, \quad \pp{\aux_{nt}}{t} + \bdiv \aux \dotprod n = 0
	\end{equation*}
	with corner conditions of the form \eqref{eq:corner_conditions},
	which appears, e.g., in boundary value problems of the Cahn-Hilliard equation. The theory presented in the following can easily be extended to mixed boundary conditions including also this fourth type.
\end{remark}

A standard (primal) variational formulation of \eqref{eq:class_formulation} is given as follows: find $\primal \in \Vprimal$ such that
\begin{equation}\label{eq:primal_formulation}
	\int_\Omega \stiffnessC \grad^2 \primal : \grad^2 \primalt \ dx = \langle F,\primalt \rangle \quad \text{for all} \ \primalt \in \Vprimal,
\end{equation}
with the Frobenius inner product $\bm{A} : \bm{B} = \sum_{i,j} \bm{A}_{ij} \, \bm{B}_{ij}$ for matrices $\bm{A}, \bm{B}$, the right-hand side $\langle F,\primalt \rangle = \int_\Omega f \, v \ dx$, and the function space
\begin{equation}\label{eq:primal_formulation_space}
	\Vprimal = \{\primalt \in H^2(\Omega) : \ \primalt = 0, \ \pp{\primalt}{n} = 0 \ \text{on} \ \Gamma_c, \quad \primalt = 0 \ \text{on} \ \Gamma_s \}
\end{equation}
with associated norm $\|\primalt\|_{\Vprimal} = \|\primalt\|_2$. Following, e.g., \cite{adams_fournier_2003,mcLean_2000}, here and throughout the paper  $L^2(\Omega)$ and $H^m(\Omega)$ denote the standard Lebesgue and Sobolev spaces of functions on $\Omega$ with corresponding norms $\|.\|_{0}$ and $\|.\|_{m}$ for positive integers $m$. For functions on $\Gamma$ we use $L^2(\Gamma)$ and $H^\frac{1}{2}(\Gamma)$ to denote the Lebesgue space and the trace space of $H^1(\Omega)$  with corresponding norms $\|.\|_{0,\Gamma}$ and $\|.\|_{\frac{1}{2},\Gamma}$. Moreover, $H^1_{0,\Gamma'}(\Omega)$ denotes the set of functions in $H^1(\Omega)$ which vanish on a part $\Gamma'$ of $\Gamma$.
The $L^2$-inner product on $\Omega$ and $\Gamma'$ are always denoted by $(.,.)$ and $(.,.)_{\Gamma'}$, respectively, no matter whether it is used for scalar, vector-valued, or matrix-valued functions.

In order to avoid technicalities related to rigid body motions, we assume throughout the paper that $\Gamma_c$ contains at least one non-trivial edge $E \in \GammaEdges$.
Then existence and uniqueness of a solution $\primal$ to \eqref{eq:primal_formulation} are guaranteed by the theorem of Lax-Milgram  (see, e.g., \cite{lion_1972, necas_2012}) for even more general right-hand sides $\langle F,\primalt\rangle$, where $F \in W^*$. Here we use $H^*$ to denote the dual of a Hilbert space $H$ and $\langle .,. \rangle$ for the duality product on $H^*\times H$. Moreover, the solution $\primal$ depends continuously on $F$
\begin{equation}\label{eq:primal_formulation_stability}
	\|\primal\|_W \leq c \, \|F\|_{W^*},
\end{equation}
with $c= c' / \lambda_{min}(\stiffnessC)$, where $c'$ depends only on the constant $c_F$ of Friedrichs' inequality. All results of this paper can easily be extended to the case $\Gamma_c = \emptyset$ under appropriate compatibility conditions for the right-hand side $F$.

\red{
For scalar functions  $\primalt$, vector-valued functions $\vphit$, and matrix-valued functions $\auxt$ the first order differential expressions 
\begin{equation*}
	\grad \primalt, \grad \vphit, \curl \primalt, \bCurl \vphit, \div \vphit, \bdiv \auxt, \rot \vphit, \brot \auxt
\end{equation*}
are defined in the weak sense on the corresponding domains of definition 
\begin{equation*}
	H^1(\Omega), (H^1(\Omega))^2, H(\curl, \Omega), H(\bCurl, \Omega), \dots .
\end{equation*}
In case that all components are in $H^1(\Omega)$ they take on their classical form given as follows:
}
\begin{align*}
	\grad \primalt &= 
	\begin{pmatrix}
	\pp{\primalt}{1} \\
	\pp{\primalt}{2}
	\end{pmatrix},\quad
	&\curl \primalt &=
	\begin{pmatrix}
	\pp{\primalt}{2} \\
	-\pp{\primalt}{1}
	\end{pmatrix},\\
	\grad \vphit &=
	\begin{pmatrix}
	\pp{\vphit_{1}}{1}  & \pp{\vphit_{1}}{2}  \\
	\pp{\vphit_{2}}{1} & \pp{\vphit_{2}}{2}  \\
	\end{pmatrix},\quad
	&\bCurl \vphit&=
	\begin{pmatrix}
	\pp{\vphit_{1}}{2} & -\pp{\vphit_{1}}{1}  \\
	\pp{\vphit_{2}}{2} & -\pp{\vphit_{2}}{1}  \\
	\end{pmatrix},\\
	\div \vphit &= \pp{\vphit_1}{1} + \pp{\vphit_2}{2}, \quad &\rot \vphit &= \pp{\vphit_2}{1} - \pp{\vphit_1}{2},\\
	\bdiv \auxt &= 
	\begin{pmatrix}
	\pp{\auxt_{11}}{1} + \pp{\auxt_{12}}{2}  \\
	\pp{\auxt_{21}}{1} + \pp{\auxt_{22}}{2}  \\
	\end{pmatrix},\quad
	&\brot \auxt &= 
	\begin{pmatrix}
	\pp{\auxt_{12}}{1} - \pp{\auxt_{11}}{2}  \\
	\pp{\auxt_{22}}{1} - \pp{\auxt_{21}}{2}  \\
	\end{pmatrix}.
\end{align*}
Moreover, the symmetric gradient and the symmetric $\bCurl$ are introduced by
\begin{equation*}
	\varepsilon (\vphit) = \frac{1}{2}( \grad \vphit + (\grad \vphit)^T), \quad \symbCurl \vphit = \frac{1}{2}( \bCurl \vphit + (\bCurl \vphit)^T).
\end{equation*}


\section{A new mixed variational formulation}
\label{sec:new_mixed_formulation}
\red{
For the new mixed variational formulation we introduce the bending moments $\aux$, as defined in \eqref{eq:aux}, as auxiliary variable. 
Then the Kirchhoff plate bending problem reads in terms of $\aux$ as stated in \eqref{eq:class_formulation}. Note, the involved operators are the second order differential operators $\grad^2$ and $\Opdivdiv$. In the following we give a formally precise definition of them.

Throughout the paper, the differential expression $\grad^2 \primalt$ is only considered for functions $\primalt \in \Vprimal \subset H^2(\Omega)$. Therefore, we define $\grad^2 \primalt$ in the standard way as the matrix consisting of all second order partial derivatives. In order to introduce the operator $\Opdivdiv$ we use the classical concept of (possibly unbounded) densely defined linear operators $B$.
Later on we consider instead of a general operator $B$ the Hessian $\grad^2$ and define $\Opdivdiv$ as its adjoint.
}

We consider an operator $B \colon D(B) \subset X \rightarrow Y^*$, where $X$ and $Y$ are Hilbert spaces and $D(B)$, the domain of definition of $B$, is dense in $X$.
The adjoint $B^* \colon \dom(B^*) \subset Y \rightarrow X^*$ is then defined as follows:
$y \in D(B^*)$ if and only if $y \in Y$ and there is a linear functional $G \in X^*$ such that 
\begin{equation}\label{eq:domainAdjoint_conditionG}
  \langle B x,y\rangle = \langle G,x\rangle \quad \text{for all} \ x \in D(B).
\end{equation}
In this case we define $B^*y = G$. Note that $\langle B^* y, x \rangle$ is well-defined for $x \in X$ and $y \in D(B^*)$ and we have in particular
\begin{equation}\label{eq:adjoint_property}
	\langle B^*y, x \rangle  = \langle B x, y \rangle \quad \text{for all} \ x \in \dom(B),\ y \in \dom(B^*).
\end{equation}
The domain $D(B^*)$ is a Hilbert space w.r.t.~the graph norm \red{$\|y\|_{D(B^*)} = (\|y\|_Y^2 + \|B^*y\|_{X^*}^2)^\frac{1}{2}$.}



\red{As already indicated above}, it is quite natural to choose $D(B) = W$ and to define $B = \nabla^2$ as an operator mapping to $Y = \LtwobSym$, (or, more precisely, to the dual of $Y$,) given by
\begin{equation*} 
  \langle \grad^2 \primalt, \auxt \rangle = \int_\Omega \grad^2 \primalt : \auxt \ dx
  \quad \text{for} \ \primalt \in D(B) = W,\ \auxt \in Y = \LtwobSym,
\end{equation*}
where $\LtwobSym$ denotes the space of symmetric matrix-valued functions given by
\begin{equation*}
	\LtwobSym = \{ \auxt : \auxt_{ij} = \auxt_{ji} \in L^2(\Omega), \ i,j = 1,2 \}
\end{equation*}
and equipped with the standard $L^2$-norm $\|\auxt\|_0$ for \red{a} matrix-valued \red{function} $\auxt$. 

\red{Keep in mind, in the following we always fix $D(B) = \Vprimal$ and obtain for the adjoint $B^*$ different domains of definition $D(B^*)$, which strongly depend on the choice of $X$.
There are several options how to choose $X$. However, according to the discussion from above, there is a restriction to meet:
$D(B)$ is a dense subset of $X$.} Now we discuss three possible choices for $X$. A first and trivial option would be $X = W$. Then it is easy to see that $D(B^*) = \LtwobSym$. Note that for this choice we have $X \subset H^2(\Omega)$, so the disadvantage for the mixed method is to work with a second-order Sobolev space for $\primal$. 
A second option would be $X = L^2(\Omega)$. Then it turns out that $D(B^*) \subset \bm{H}(\Opdivdiv, \Omega)_\text{sym} = \{ \auxt \in \LtwobSym : \Opdivdiv \auxt \in L^2(\Omega) \}$, where here $\Opdivdiv$ is defined in the distributional sense. This time the disadvantage for the mixed method is to work with a second-order Sobolev space for $\aux$.

The idea for the new mixed formulation is to distribute the smoothness requirements evenly among $\primal$ and $\aux$ by choosing the space $X$ as an intermediate space
between $\Vprimal$ and $L^2(\Omega)$. In particular, we propose to set $X$ equal to $Q$, given  by
\begin{equation*}
	\Q = H^1_{0,\Gamma_c \cup \Gamma_s}(\Omega) = \{\primalt \in H^1(\Omega) : \primalt = 0 \ \text{on} \ \Gamma_c \cup \Gamma_s \},
\end{equation*}
equipped with the norm $\|\primalt\|_\Q = \|\primalt\|_1$. 
\begin{remark}
	The space $\Q$ is the interpolation space between $\Vprimal$ and $L^2(\Omega)$. Note that we only use interpolation as motivation, but do not rely in the following on results from interpolation theory.
\end{remark}

This choice for $X$ meets the required condition:
\begin{lemma} \label{lem:VdenseQ}
	The subspace $\Vprimal$ is dense in $\Q$.
\end{lemma}
\begin{proof}
	We follow the lines of the proof of \cite[Theorem 1.6.1]{grisvard_1992}. In order to verify the density, we have to check that the trace space $\gamma(\Vprimal)$ is dense in the trace space $\gamma(\Q)$, where $\gamma$ is the standard $H^1$-trace representing the value on the boundary. By considering the situation locally near each corner in $\GammaVertices$, the required density follows from the density of $C^\infty_0(\mathbb{R}^+)$ in $H^{\frac{1}{2}}(\mathbb{R}^+)$ and $\tilde H^{\frac{1}{2}}(\mathbb{R}^+)$; see \cite{grisvard_1985}.
\end{proof}

Now we leave the abstract framework and use, from now on, the notations
\begin{equation}\label{eq:operator_Divdiv}
  \div\bdiv \quad \text{and} \quad \Hdivdiv{\Q^*}
\end{equation}
instead of $B^*$  and $D(B^*)$, respectively, \red{with $X = \Q$ and unchanged $D(B) = \Vprimal$ and $Y = \LtwobSym$}. In consistence with the abstract framework, the Hilbert space $\Hdivdiv{\Q^*}$ is explicitly given by
\begin{equation}\label{eq:domain_Divdiv}
\begin{alignedat}{1}
	&\Hdivdiv{\Q^*} 
	= \{ \auxt \in \LtwobSym \colon \ \text{the functional} \\
	&\qquad G \colon \primalt \mapsto \int_\Omega \nabla^2 \primalt : \auxt \  d x, \ \primalt \in \Vprimal, 
	\ \text{is bounded w.r.t.~the } \Q\text{-norm} \},
\end{alignedat}
\end{equation}
equipped with the norm $\normHdivdiv{\auxt}{\Q^*} = (\|\auxt\|^2_0 + \|\Opdivdiv \auxt\|^2_{\Q^*})^{1/2}$.

This motivates the new mixed formulation as follows: For $F\in \Q^*$, find $\aux \in \V$ and $\primal \in \Q$ such that
\begin{equation}\label{eq:newMixed_formulation}
\begin{alignedat}{4} 
	 & (\aux, \auxt)_{\stiffnessC^{-1}} & & + \langle\Opdivdiv \auxt, \primal \rangle & & = 0 & \quad & \text{for all} \ \auxt \in \V, \\
	 & \langle\Opdivdiv \aux, \primalt \rangle & & & & = -\langle F, \primalt \rangle & \quad & \text{for all} \ \primalt \in \Q,
\end{alignedat}
\end{equation}
with the function spaces
\begin{equation}\label{eq:newMixed_formulation_spaces}
	\V = \Hdivdiv{\Q^*}, \quad
	\Q = H^1_{0,\Gamma_c \cup \Gamma_s}(\Omega),
\end{equation}
equipped with the norms $\|\auxt\|_\V = \normHdivdiv{\auxt}{\Q^*}$ and $\|\primalt\|_\Q = \|\primalt\|_1$.
Here, we use the notation $(\aux, \auxt)_{\stiffnessC^{-1}} = (\stiffnessC^{-1} \aux, \auxt)$.

The first line in \eqref{eq:newMixed_formulation} comes from the relation $\stiffnessC^{-1} \aux + \nabla^2 \primal = 0$ between bending moment $\aux$ and deflection $\primal$, the second line originates from \eqref{eq:class_formulation}. Note that we require additional regularity of $F$, namely $F\in \Q^*$. In contrast, for the primal problem \eqref{eq:primal_formulation} we need only $F\in \Vprimal^*$.

\begin{remark}
\red{The operator $\Opdivdiv$ as defined in \eqref{eq:operator_Divdiv} and the composition of the first order operators $\div$ and $\bdiv$ introduced at the end of \sectref{sec:kirchhoff_love_plate} differ in two ways. First of all, their domains of definition are different. While $\div(\bdiv \auxt)$ is only well-defined for functions $\auxt \in \Ltwob$, where $\div \auxt$ and $\div\bdiv \auxt$ are $L^2$-functions as well, the domain of definition of $\Opdivdiv$ in \eqref{eq:domain_Divdiv} contains functions $\auxt\in\LtwobSym$ with the less restrictive requirement $\Opdivdiv\auxt \in \Q^*$. The domain of definition of $\Opdivdiv$ in \eqref{eq:domain_Divdiv} includes the boundary conditions $\auxt_{nn} = 0$ on $\Gamma_s \cup \Gamma_f$ and $\vertexJump{\auxt_{nt}}_{x} = 0$ for all $x \in \GammaVertices[f]$ in a weak sense, as we show later in \thmref{thm:spaceChar_Hdivdiv1_smooth}.

But even on the intersection of the domains of definition $\div(\bdiv \auxt)$ coincides with $\Opdivdiv \auxt$, only if $\auxt$ satisfies the boundary condition $\pp{\auxt_{nt}}{t} + \bdiv \auxt \dotprod n = 0$ on $\Gamma_f$.}
\end{remark}

\begin{remark}
	In \cite{sinwel_2009,pechstein_schoeberl_2011,pechstein_schoeberl_2016} a similar nonstandard Sobolev space is introduced. Note, our way of definition is different and well-suited for the further considerations.
\end{remark}

Problem \eqref{eq:newMixed_formulation} has the typical structure of a saddle point problem
\begin{equation*} 
\begin{alignedat}{4}
  & a(\aux,\auxt) & & + b(\auxt,\primal) & & = 0  & \quad & \text{for all} \ \auxt \in \V , \\
  & b(\aux,\primalt) & & & & = - \langle F , \primalt \rangle & \quad & \text{for all} \ \primalt \in \Q,
\end{alignedat}
\end{equation*}
whose associated linear operator $\mathcal{A} \colon \V \times \Q \longrightarrow (\V \times \Q)^*$ is given by
\[
  \left\langle 
     \mathcal{A} 
     (\aux, \primal),
     (\auxt, \primalt) 
  \right\rangle
    = a(\aux,\auxt) + b(\auxt,\primal) + b(\aux,\primalt).
\]
If the bilinear form $a$ is symmetric, i.e., $a(\aux,\auxt) = a(\auxt,\aux)$, and non-negative, i.e., $a(\auxt,\auxt) \ge 0$, which is fulfilled for \eqref{eq:newMixed_formulation}, it is well-known that $\mathcal{A}$ is an isomorphism from $\V \times \Q$ onto $(\V \times \Q)^*$, if and only if the following conditions are satisfied; see, e.g., \cite{boffi_brezzi_fortin_2013}:
\begin{enumerate}
\item 
$a$ is bounded: There is a constant $\|a\| > 0$ such that
\[
  |a(\aux,\auxt)| \le \|a\| \, \|\aux\|_{\V} \, \|\auxt\|_{\V} \quad \text{for all} \ \aux, \ \auxt \in \V.
\]
\item 
$b$ is bounded: There is a constant $\|b\| > 0$ such that
\[
  |b(\auxt,\primalt)| \le \|b\| \, \|\auxt\|_{\V} \|\primalt\|_\Q \quad \text{for all} \ \auxt \in \V, \ \primalt \in \Q.
\]
\item 
$a$ is coercive on the kernel of $b$: There is a constant $\alpha > 0$ such that
\[
  a(\auxt,\auxt) \ge \alpha \, \|\auxt\|_{\V}^2 \quad \text{for all} \ \auxt \in \ker B
\]
with $\ker B = \{ \auxt \in \V \colon b(\auxt,\primalt) = 0 \quad \text{for all} \ \primalt \in \Q \}$.
\item 
$b$ satisfies the inf-sup condition: There is a constant $\beta > 0$ such that
\[
  \inf_{\rule[0.6ex]{0ex}{1ex} 0 \neq \primalt \in \Q} \sup_{0 \neq \auxt \in \V} 
  \frac{b(\auxt,\primalt)}{\|\auxt\|_{\V} \, \|\primalt\|_\Q} \ge \beta.
\]
\end{enumerate}
We will refer to these conditions as Brezzi's conditions with constants $\|a\|$, $\|b\|$, $\alpha$, and $\beta$.

In order to verify Brezzi's conditions for \eqref{eq:newMixed_formulation}, we need the following result on the relation between the primal problem \eqref{eq:primal_formulation} and the new mixed problem \eqref{eq:newMixed_formulation}.
\begin{theorem}
\label{thm:relation_primal_mixed}
	Let $\primal$ be the solution of the primal problem \eqref{eq:primal_formulation} for $F\in \Q^*$. Then we have $\aux = -\stiffnessC \grad^2 \primal \ \in \V$ and $(\aux, \primal)$ solves the mixed problem \eqref{eq:newMixed_formulation}.
\end{theorem}
\begin{proof}
	Since $\primal \in W$ solves \eqref{eq:primal_formulation}, it follows that $\red{\aux} \in \LtwobSym$ and
	\begin{equation*}
		\langle \grad^2 \primalt, \aux \rangle = \int_\Omega \aux : \grad^2 \primalt \ dx = -\int_\Omega \stiffnessC \grad^2 \primal : \grad^2 \primalt \ dx = \langle -F,\primalt \rangle \quad \text{for all} \ \primalt \in \Vprimal.
	\end{equation*}
	From the definition of the domain of the adjoint operator in \eqref{eq:domainAdjoint_conditionG} we obtain 	that $\div\bdiv \aux = - F\in \Q^*$, which shows that $\aux \in \V$ and that the second equation of \eqref{eq:newMixed_formulation} is satisfied.
    Using \eqref{eq:adjoint_property}, we receive
	\begin{equation*}
		\langle \Opdivdiv \auxt, \primal \rangle = \langle \grad^2 \primal, \auxt \rangle = \int_\Omega \auxt : \grad^2 \primal \ dx = -\int_\Omega \auxt : \stiffnessC^{-1}\aux \ dx
	\end{equation*}
	for all $\auxt \in \V$,
	which proves the first equation.
\end{proof}

\begin{theorem}\label{thm:newMixed_formulation_brezziCond}
	The mixed problem defined by \eqref{eq:newMixed_formulation} and \eqref{eq:newMixed_formulation_spaces} 
	satisfies Brezzi's conditions with the constants $\|a\| = 1 / \lambda_{min}(\stiffnessC)$, $\|b\|=1$, $\alpha = 1/ \lambda_{max}(\stiffnessC)$ and $\beta = (1 + c)^{-1/2}$, where $c = c' \, \lambda_{max}(\stiffnessC) / \lambda_{min}(\stiffnessC)$ and $c'$ as in \eqref{eq:primal_formulation_stability}.
\end{theorem}
\begin{proof}
The verification of the first three parts of Brezzi's conditions is simple and, therefore, omitted.
For showing the inf-sup condition,
	let $\primal^\primalt$ be the solution of the primal problem \eqref{eq:primal_formulation} with the right-hand side $F^\primalt = - (\primalt,.)_Q \in Q^*$ for a fixed but arbitrary $\primalt \in \Q$.
	From \thmref{thm:relation_primal_mixed} it follows that $\aux^\primalt = -\stiffnessC \grad^2 \primal^\primalt \ \in \V$,
	and $(\aux^\primalt, \primal^\primalt)$ is solution of the corresponding mixed problem \eqref{eq:newMixed_formulation}.
	From the second line of the mixed formulation \eqref{eq:newMixed_formulation} we obtain 
	\begin{equation*}
		\langle \Opdivdiv \aux^\primalt, \primalt \rangle = (\primalt, \primalt)_Q = \|\primalt\|^2_\Q
	\end{equation*}
	and
	\begin{equation*}
		\|\Opdivdiv \aux^\primalt\|_{\Q^*} =  \sup_{\vpt \in \Q} \frac{ \langle \Opdivdiv \aux^\primalt, \vpt \rangle}{\|\vpt\|_Q} = \sup_{\vpt \in \Q} \frac{ (\primalt, \vpt)_Q }{\|\vpt\|_\Q} = \|\primalt\|_\Q.
	\end{equation*}
	Using the stability estimate \eqref{eq:primal_formulation_stability} we obtain
	\begin{align*}
		\|\aux^\primalt\|^2_0 &= \|\stiffnessC \grad^2 \primal^\primalt\|^2_0 \leq \lambda_{max}(\stiffnessC) (\stiffnessC \grad^2 \primal^\primalt,\grad^2 \primal^\primalt)  = \lambda_{max}(\stiffnessC)  \langle F^\primalt,\primal^\primalt\rangle\\
		& \leq \lambda_{max}(\stiffnessC) \|F^\primalt\|_{W^*} \|\primal^\primalt\|_W \leq c \ \|F^\primalt\|_{W^*}^2 \leq c \ \|F^\primalt\|^2_{\Q^*} = c \ \|\primalt\|^2_\Q
	\end{align*}
	with $c = c' \, \lambda_{max}(\stiffnessC) / \lambda_{min}(\stiffnessC)$. Hence,
	\[
	  \normHdivdiv{\aux^\primalt}{\Q^*}^2 
	   = \|\aux^\primalt\|^2_0 + \|\Opdivdiv \aux^\primalt \|^2_{\Q^*}
	   \leq (1+c)\, \|\primalt\|_\Q^2.
	\]
	Therefore,
	\begin{equation*}
		\sup_{0\neq\auxt\in \V} \frac{\langle \Opdivdiv \auxt, \primalt \rangle}{\normHdivdiv{\auxt}{\Q^*}} 
		\geq \frac{\langle \Opdivdiv \aux^\primalt, \primalt \rangle}{\normHdivdiv{\aux^\primalt}{\Q^*}}
		\geq (1+c)^{-1/2} \ \|\primalt\|_\Q ,
	\end{equation*}
	which completes the proof.
\end{proof}

\begin{remark} The choice of $\aux^v$ for proving the inf-sup condition is different to the choice $\aux^v = v \, \Id$ as it is used, e.g., in \cite{brezzi_raviart_1977,krendl_rafetseder_zulehner_2016}, which would not work here, since $v \, \Id \notin \V$ for problems with a free boundary part. 
\end{remark}

\begin{corollary}\label{cor:equivalence}
	\red{For $F \in \Q^*$}, the primal problem \eqref{eq:primal_formulation} and the mixed problem \eqref{eq:newMixed_formulation} are equivalent in the following sense: If $\primal$ solves \eqref{eq:primal_formulation}, then $\aux = -\stiffnessC \grad^2 \primal \ \in \V$ and $(\aux, \primal)$ solves \eqref{eq:newMixed_formulation}. Vice versa, if $(\aux, \primal)$ solves \eqref{eq:newMixed_formulation}, then $\primal \in \Vprimal$ and solves \eqref{eq:primal_formulation}.
\end{corollary}
\begin{proof}
	The first part has already been shown in \thmref{thm:relation_primal_mixed}. Since, both problems are uniquely solvable the reverse direction is true as well.
\end{proof}

In the case of a purely clamped plate, we have $Q = H_0^1(\Omega)$. 
So, in this mixed setting,  $\primal = 0$ on $\Gamma$ is treated as an essential boundary condition, while $\pp{\primal}{n} = 0$ on $\Gamma$ becomes a natural boundary condition incorporated in the variational formulation.
No boundary conditions are prescribed for $\aux$, which makes the definition of an appropriate space $\V$ for $\aux$ much easier. The space $\V$ can be introduced directly as $\{ \auxt \in \LtwobSym \colon \Opdivdiv \in H^{-1}(\Omega) \}$, where here $\Opdivdiv$ is to be interpreted in the distributional sense, see \cite{krendl_rafetseder_zulehner_2016}. 

The situation is more involved for mixed boundary conditions. Here we have $\Q = H^1_{0,\Gamma_c \cup \Gamma_s}(\Omega)$. So  $\primal = 0$ on $\Gamma_c \cup \Gamma_s$ is treated as an essential boundary condition, while $\pp{\primal}{n}=0$ on $\Gamma_c$ and $\pp{\aux_{nt}}{t} + \bdiv \aux \dotprod n$ on $\Gamma_f$ become natural boundary conditions incorporated in the variational formulation. 
The remaining boundary conditions $\aux_{nn} = 0$ on $\Gamma_s \cup\Gamma_f$ and the corner conditions (\ref{eq:corner_conditions}) are treated as essential boundary conditions: They are not explicitly visible but are hidden in the definition of the space $\Hdivdiv{\Q^*}$.
We doubt that there is an easy and direct way of formulating the corner conditions \eqref{eq:corner_conditions} with the help of pointwise trace operators in $\Hdivdiv{\Q^*}$. 
However, for sufficiently smooth functions $\auxt$, the corner conditions can be explicitly extracted using the definition of $\Hdivdiv{\Q^*}$, as we will see in the next theorem.

\begin{theorem}\label{thm:spaceChar_Hdivdiv1_smooth}
	Let $\auxt \in \LtwobSym \cap \bm{C}^1(\overline\Omega)$.
	Then $\auxt \in \Hdivdiv{\Q^*}$ if and only if
	\begin{equation}\label{eq:bc_in_Hdivdiv}
		\auxt_{nn} = 0 \quad \text{on} \ \Gamma_s \cup \Gamma_f
		\quad \text{and} \quad
		\vertexJump{\auxt_{nt}}_{x} = 0 \quad \text{for all} \ x \in \GammaVertices[f].
	\end{equation}
\end{theorem}
\begin{proof}
	Recall the representation of $\Hdivdiv{\Q^*}$ in \eqref{eq:domain_Divdiv}. For $\primalt \in \Vprimal$, we obtain by integration by parts
	\begin{equation}\label{eq:g_1}
	\begin{aligned}
	\langle G,v \rangle 
	&= \int_\Omega \auxt:\grad^2 \primalt \ dx\\
	& \red{= -\int_\Omega \bdiv \auxt \dotprod \grad \primalt \ dx + \int_{\Gamma_s \cup \Gamma_f} \auxt_{nn} \, \pp{\primalt}{n} \ ds} \\
	& \quad - \int_{\Gamma_f} (\pp{\auxt_{nt}}{t}) \,  \primalt \ ds + \sum_{x \in \GammaVertices[f]} \vertexJump{\auxt_{nt}}_{x} \primalt(x).
	\end{aligned}
	\end{equation}
	Assume now that $\auxt$ satisfies the boundary conditions (\ref{eq:bc_in_Hdivdiv}). Then we have
	\begin{equation*}
		\langle G,v \rangle  = -\int_\Omega \bdiv \auxt \dotprod \grad \primalt \ dx - \int_{\Gamma_f} (\pp{\auxt_{nt}}{t}) \, \primalt \ ds,
	\end{equation*}
	which is obviously bounded w.r.t.~the $H^1$-norm. Hence $\auxt \in \Hdivdiv{\Q^*}$.
	
	On the other hand, if $\auxt \in \Hdivdiv{\Q^*}$, then the functional $G$ given by (\ref{eq:g_1}) is bounded w.r.t.~the $H^1$-norm. For \red{$v \in \HtwoNV \cap \Vprimal$} we obtain
	\begin{equation}\label{eq:g_2}
	  \langle G,v \rangle = -\int_\Omega \bdiv \auxt \dotprod \grad \primalt \ dx - \int_{\Gamma_f} (\pp{\auxt_{nt}}{t}) \, \primalt \ ds
	\end{equation}
	where $\HtwoNV = \{\primalt \in H^2(\Omega) : \pp{\primalt}{n} = 0 \ \text{on} \ \Gamma, \quad \primalt(x) = 0 \ \text{for all} \ x \in \GammaVertices \}$.
	Note that all expressions in (\ref{eq:g_2}) are continuous in $v$ w.r.t.~the $H^1$-norm. 
	Analoguously to the proof of \lemref{lem:VdenseQ}, one can show that $\HtwoNV \cap \Vprimal$ is dense in $\Vprimal$ w.r.t.~the $H^1$-norm. Then it follows that (\ref{eq:g_2}) is valid for all $v \in \Vprimal$.  This implies together with (\ref{eq:g_1}) that
	\begin{equation*}
		\int_{\Gamma_s \cup \Gamma_f} \auxt_{nn} \pp{\primalt}{n} \ ds + \sum_{x \in \GammaVertices[f]} \vertexJump{\auxt_{nt}}_{x} \primalt(x) = 0 \quad \text{for all} \ \primalt \in \Vprimal.
	\end{equation*}
	From this the boundary conditions in (\ref{eq:bc_in_Hdivdiv}) follow by standard arguments.
\end{proof}
%

\begin{remark}
A similar characterization can be derived for piecewise smooth\linebreak functions from $\Hdivdiv{\Q^*}$, e.g., for functions from finite element spaces.
\end{remark}


\section{Regular decomposition}
\label{sec:regular_decomposition}

The rather simple proof of the well-posedness of the new mixed formulation \eqref{eq:newMixed_formulation} comes at the cost of the nonstandard Sobolev space $\Hdivdiv{\Q^*}$.
The next theorem provides a regular decomposition of this space, which makes $\Hdivdiv{\Q^*}$ computationally accessible.
In order to derive our decomposition, we need a characterization of the kernel of the distributional $\Opdivdiv$, given in the next lemma, see \cite{huang_huang_xu_2011,krendl_rafetseder_zulehner_2016} for a proof.
\begin{lemma}
\label{lem:kernel}
	Let $\Omega$ be simply connected. For $\auxt \in \LtwobSym$, we have $\Opdivdiv \auxt = 0$ 
	in the distributional sense iff $\auxt = \symbCurl \vphit$ for some function $\vphit\in (H^1(\Omega))^2$.
	The function $\vphit$ is unique up to an element from $\text{RT}_0 = \{ a x + b \colon a \in \mathbb{R}, \ b\in \mathbb{R}^2\}$.
\end{lemma}

In preparation for the next theorem, observe that $\bdiv \bCurl \vphit = 0$. Therefore, $\bCurl \vphit \in \Hdiv = \{\auxt \in \Ltwob: \bdiv \auxt \in (L^2(\Omega))^2 \}$, the normal component of $\bCurl \vphit$ is well-defined on the boundary
\begin{equation*}
	(\bCurl \vphit) n \in (H^{-\frac{1}{2}}(\Gamma))^2 
	\quad \text{with} \quad H^{-\frac{1}{2}}(\Gamma) = (H^{\frac{1}{2}}(\Gamma))^*,
\end{equation*}
and we have by integration by parts
\begin{equation*}
	\int_\Omega \bCurl \vphit : \grad \vphitt 
	= \langle (\bCurl \vphit) n, \vphitt \rangle_\Gamma,
\end{equation*}
for all $\vphit \in (H^1(\Omega))^2$ and $\vphitt \in (H^1(\Omega))^2$. 
For smooth functions $(\bCurl \vphit) n$ coincides with the tangential derivative of $\vphit$, so we use the notation
\begin{equation*}
	 \pp{\vphit}{t} = (\bCurl \vphit) n.
\end{equation*}

\begin{theorem}
\label{thm:regular_decomposition}
	Let $\Omega$ be simply connected. For each $\auxt \in \Hdivdiv{\Q^*}$ there exists a decomposition
	\begin{equation}\label{eq:Mpphi}
		\auxt= \vpt \Id + \symbCurl \vphit
	\end{equation}
	with $\vpt\in \Q = H^1_{0,\Gamma_c \cup \Gamma_s}(\Omega)$ and $\vphit \in (H^1(\Omega))^2$ satisfying the coupling condition
	\begin{equation}\label{eq:coupling_cond}
		\langle \pp{\vphit}{t} ,\grad \primalt \rangle_\Gamma = - \int_\Gamma \vpt \ \pp{\primalt}{n} \ ds \quad \text{for all} \ \primalt \in \Vprimal .
	\end{equation}
	The function $\vpt \in Q$ is the unique solution of the Poisson problem
	\begin{equation}\label{PoissonNq}
		 \int_\Omega \grad \vpt\dotprod \grad \primalt \ dx \quad = - \langle \Opdivdiv \auxt, \primalt \rangle \quad \text{for all} \ \primalt \in \Q,
	\end{equation}
	and $\vphit \in (H^1(\Omega))^2$ is unique up to an element from $\text{RT}_0$. Vice versa, for each $\auxt$ given by \eqref{eq:Mpphi} with $\vpt \in Q$ and $\vphit \in (H^1(\Omega))^2$ satisfying \eqref{eq:coupling_cond}, it follows that $\auxt \in \Hdivdiv{\Q^*}$ \red{with $\langle \Opdivdiv \auxt, \primalt \rangle = -\int_\Omega \grad \vpt\dotprod \grad \primalt \ dx$ for all $\primalt \in \Q$}.
	Moreover, 
	\begin{equation}\label{eq:norm_equivalence}
		\underline{c} \ ( \|\vpt\|^2_1 + \|\symbCurl\vphit\|_0^2) \leq \normHdivdiv{\auxt}{\Q^*}^2 \leq \overline{c} \ ( \|\vpt\|^2_1 + \|\symbCurl\vphit\|_0^2),
	\end{equation}
	with positive constants $\underline{c}$ and $\overline{c}$, which depend only on the constant $c_F$ of Fried\-richs' inequality.
\end{theorem}
\begin{proof}
	Let $\vpt\in \Q$ be the unique solution of the variational problem
	\begin{equation}\label{eq:Poisson_q}
		\int_\Omega \grad \vpt\dotprod \grad \primalt = -\langle \Opdivdiv{\auxt}, \primalt \rangle \quad \text{for all} \ \primalt \in \Q.
	\end{equation}
	For $\primalt \in C^\infty_0(\Omega)$ we receive from integration by parts
	\begin{equation*}
		\langle \Opdivdiv \vpt\Id, \primalt \rangle = \int_\Omega \vpt\Id:\grad^2 \primalt \ dx = -\int_\Omega \grad \vpt\dotprod \grad \primalt \ dx.
	\end{equation*}
	This implies $\Opdivdiv (\auxt - \vpt\Id) = 0$ in the distributional sense. According to \lemref{lem:kernel}, there exists a function $\vphit \in (H^1(\Omega))^2$ such that
	\begin{equation*}
		\auxt - \vpt \Id = \symbCurl \vphit.
	\end{equation*}
	For $\primalt \in \Vprimal$ integration by parts provides
	\begin{align*}
		\langle \Opdivdiv \auxt, \primalt \rangle 
		&= \langle \grad^2 \primalt, \auxt \rangle
		\red{=} -\int_\Omega \grad \vpt\dotprod \grad \primalt \ dx + \int_\Gamma \vpt \ \pp{\primalt}{n} \ ds + \langle \pp{\vphit}{t} ,\grad \primalt \rangle_\Gamma.
	\end{align*}
	With \eqref{eq:Poisson_q} it follows that
	\begin{equation*}
	 	\langle \pp{\vphit}{t},\grad \primalt \rangle_\Gamma = - \int_\Gamma \vpt \ \pp{\primalt}{n} \ ds \quad \text{for all} \ \primalt \in \Vprimal.
	\end{equation*}
	For the reverse direction, assume that (\ref{eq:Mpphi}) and (\ref{eq:coupling_cond}) hold. By using the integration by parts formula from above we obtain 
	\begin{equation*}
		\langle \grad^2 \primalt, \auxt \rangle =   -\int_\Omega \grad \vpt\dotprod \grad \primalt \ dx,
	\end{equation*}
	\red{where \eqref{eq:coupling_cond} makes the boundary contributions vanish.} This immediately implies that $G\colon v \mapsto \langle \grad^2 \primalt, \auxt \rangle$ is bounded w.r.t.~the $\Q$-norm, i.e.~$\auxt \in \Hdivdiv{\Q^*}$ \red{with $\Opdivdiv \auxt = G \in \Q^*$.}
	
	In order to show \eqref{eq:norm_equivalence}, note that \eqref{PoissonNq} implies
	\begin{equation*}
		\| \Opdivdiv \auxt \|_{\Q^*} = \sup_{\red{\primalt \in \Q}} \frac{\int_\Omega \grad \vpt \dotprod \grad \primalt \ dx}{\| \primalt\|_1}.
	\end{equation*}
	Hence,
	\begin{equation*}
		(1 + c_F^2)^{-1/2}\ |\vpt|_1 \leq \| \Opdivdiv \auxt \|_{\Q^*} \leq |\vpt |_1,
	\end{equation*}
	using Friedrichs' inequality $\|\vpt\|_0 \le c_F \, |\vpt|_1$, where $|\vpt|_1$ denotes the $H^1$-semi-norm.

	With these inequalities we obtain for the estimate from above 
	\begin{align*}
		\normHdivdiv{\auxt}{\Q^*}^2 
		  &= \|\auxt \|_0^2 + \| \Opdivdiv \auxt \|^2_{\Q^*}= \|\vpt \, \Id + \symbCurl \vphit \|_0^2 + \| \Opdivdiv \auxt \|^2_{\Q^*}\\
		  &\leq 2 \, \|\vpt \Id \|^2_0 + 2 \, \|\symbCurl \vphit \|^2_0 + |\vpt|_1^2
		   \leq 4 \, \|\vpt\|^2_1 + 2 \, \|\symbCurl \vphit \|^2_0,
	\end{align*}
	and for the estimate from below
	\begin{align*}
		& \|\vpt\|_1^2 + \|\symbCurl \vphit \|^2_0 
        = \|\vpt\|_1^2 + \|\auxt - \vpt \Id\|_0^2 \\
		& \quad \leq \|\vpt\|^2_1 + 2 \, \|\auxt \|_0^2  + 2 \, \| \vpt \, \Id\|_0^2
		    \leq 2 \, \|\auxt\|_0^2 + (1+ 5c_F^2) |\vpt|_1^2 \\
        & \quad \leq 2 \, \|\auxt\|_0^2 + (1+ 5c_F^2)(1+c_F^2) \, \|\Opdivdiv \auxt\|_{\Q^*}^2 .
	\end{align*}
	Therefore, \eqref{eq:norm_equivalence} holds with $1/\underline{c} = \max (2,(1+ 5c_F^2)(1+c_F^2))$ and $\overline{c} = 4$.
\end{proof}

\begin{remark}
	By applying Korn's inequality to $\vphit^\perp = (-\vphit_2, \vphit_1)^T$ we obtain 
	\begin{equation*}
		\red{\| \symbCurl \vphit \|_0} = \|\varepsilon (\vphit^\perp) \|_0 
		\ge c_K \, |\vphit^\perp|_1 = c_K \, |\vphit|_1
		\quad \text{for all} \ \vphit \in (H^1(\Omega))^2/RT_0,
	\end{equation*}
	where $H/RT_0$ denotes the $L^2$-orthogonal complement of $RT_0$ in $H$ for spaces $H\subset (H^1(\Omega))^2$. Then it follows that
	\[
 	  c_K (1+c_F^2)^{-1/2} \, \|\vphit\|_1 \le \| \symbCurl \vphit \|_0 \leq \|\vphit\|_1.
	\]
	So, provided the unique element $\vphit \in (H^1(\Omega))^2/RT_0$ is chosen for the decomposition, stability follows from \eqref{eq:norm_equivalence} in standard $H^1$-norms.
\end{remark}

\subsection{The coupling condition}
\label{sec:coupling_condition}

\thmref{thm:regular_decomposition} shows that each function $\auxt \in \V = \Hdivdiv{\Q^*}$ can be represented by a pair of functions $(\vpt,\vphit) \in V$ with
\begin{equation*}
  \VRegDecomp = \{ (\vpt,\vphit) \in Q \times (H^1(\Omega))^2 \colon
         \vpt, \vphit \ \text{satisfy the coupling condition}\  \eqref{eq:coupling_cond}  \}
\end{equation*}
Observe that \eqref{eq:coupling_cond} involves only traces of $\vpt$ and $\vphit$ on $\Gamma$.
We obviously have
\begin{equation*}
	\VRegDecomp = \{(\vpt, \vphit) \in Q \times (H^1(\Omega))^2: \vphit \in \SpPsi_\vpt \},
\end{equation*}
where $\SpPsi_\vpt$ is given by the following definition:
\begin{definition}\label{def:SpPsi_q}
	For fixed $\vpt \in \Q$ we define
	\begin{equation*}
		\SpPsi_\vpt = \{ \vphit \in (H^1(\Omega))^2: \langle \pp{\vphit}{t}, \grad \primalt \rangle =  -\int_\Gamma \vpt \ \pp{\primalt}{n} \ ds \quad \text{for all} \ \primalt \in \Vprimal\}.
	\end{equation*}
\end{definition}
So $\SpPsi_\vpt$ consists of those functions $\vphit \in (H^1(\Omega))^2$ that fulfill \eqref{eq:coupling_cond}, which becomes in this context a boundary condition for $\vphit$ with given $\vpt$. 
In particular, functions from $\SpPsi_0$, the space associated to $\vpt=0$, satisfy the corresponding homogeneous boundary conditions.

It is essential for deriving a decoupled formulation that $\VRegDecomp$ is a direct sum of two subspaces, which we will show next.
For this we construct, for each $\vpt \in \Q$, a specific function $\extPsi{\vpt} \in \SpPsi_\vpt$ as follows: 
Let $E$ be a fixed edge on $\Gamma_c$ with boundary $\partial E = \{x_A,x_B\}$, where $x_A$, $x_B$ are ordered in counterclockwise direction. We set
\begin{equation}\label{extensionGammaE}
	\extPsiGamma{\vpt}(x) = - \int_0^{\sigma} \vpt \, n \ ds 
	\quad \text{with} \quad x = \gamma(\sigma) 
	\quad \text{on} \ \Gamma  \setminus E.
\end{equation}
Here $(\sigma,\gamma(\sigma))$ denotes the arc length parametrization of $\Gamma$ in counterclockwise direction with $x_B = \gamma(0)$ and $x_A = \gamma(\sigma_E)$ for some $\sigma_E > 0$.
Next we extend $\vphit_\Gamma[\vpt]$ on the whole boundary $\Gamma$ by connecting its values at $\partial E$ linearly on $E$. By this $\vphit_\Gamma[\vpt]$ becomes
a continuous function on $\Gamma$ with a weak tangential derivative in $(L^2(\Gamma))^2$ satisfying
\begin{equation*}
	\pp{\extPsiGamma{\vpt}}{t} = -\vpt \, n 
	\quad \text{on} \ \Gamma  \setminus E.
\end{equation*}
Finally, let $\extPsi{\vpt} \in (H^1(\Omega))^2$ be the harmonic extension of $\extPsiGamma{\vpt}$, which is well-defined, since we obviously have $\extPsiGamma{\vpt} \in (H^\frac{1}{2}(\Gamma))^2$. 
\begin{lemma}
For each $\vpt \in Q$, let $\extPsi{q} \in \red{(H^1(\Omega))^2}$ be given as described above. Then $\extPsi{\vpt} \in \SpPsi_\vpt$ and we have
\begin{equation}\label{eq:VRegDecomp}
	\VRegDecomp = \VRegDecomp_0 \oplus \VRegDecomp_1 
	\quad \text{with} \quad
	\VRegDecomp_0 =  \{0\} \times \SpPsi_0,\quad
	\VRegDecomp_1 =  \{(\vpt, \extPsi{\vpt}) : \vpt \in \Q \}.
\end{equation}
Moreover, there is a constant $c>0$ such that
\begin{equation}\label{eq:extension_stability}
	\| \extPsi{\vpt} \|_1 \leq c \|\vpt\|_1 
	\quad \text{for all} \ q \in Q.
\end{equation}
\end{lemma}
\begin{proof}
For all $\primalt \in \Vprimal$, we have
\begin{equation*}
	\langle \pp{\extPsi{\vpt}}{t} ,\grad \primalt \rangle_\Gamma = -\int_{\Gamma \setminus E} \vpt n \dotprod \grad \primalt \ ds = -\int_\Gamma \vpt \, \pp{\primalt}{n} \ ds,
\end{equation*}
since $\grad \primalt = 0$ on $E$, which shows that $\extPsi{\vpt} \in \SpPsi_\vpt$.
The decomposition \eqref{eq:VRegDecomp} can easily be derived from the representation
\[
  (\vpt,\vphit) = (\vpt,\extPsi{\vpt}) + (0,\vphit_0) 
  \quad \text{with} \quad \vphit_0 = \vphit - \extPsi{\vpt} \in \SpPsi_0.
\]
The sum is direct, since $\extPsi{q} = 0$ for $q=0$.
Finally, for showing \eqref{eq:extension_stability}, observe that
\[ 
   \extPsiGamma{q}(x) = \big( 1 - (\sigma - \sigma_E)/|E| \big) \, \extPsiGamma{q}(x_A) 
   \quad \text{with} \quad x = \gamma(\sigma) 
    \quad \text{on} \ E,
\]
\red{where in the counterclockwise parametrization of $E$ we have $x_A = \gamma(\sigma_E)$ and $x_B = \gamma(\sigma_E + |E|)$} and $|E|$ denotes the length of $E$. From this and \eqref{extensionGammaE} it easily follows that
\[
  \|\extPsiGamma{q}\|_{1,\Gamma} \le c \, \|q\|_{0,\Gamma}.
\]
Here $H^1(\Gamma)$ denotes the space of functions with weak tangential derivative with corresponding norm $\|.\|_{1, \Gamma}$, see, e.g., \cite{mcLean_2000}. The rest follows from standard trace and inverse trace inequalities.
\end{proof}

\begin{remark}
The above construction relies on the assumption that we have at least one clamped edge, which we made at the beginning of this paper. However, this construction can be extended for general mixed boundary conditions using the freedom we have in the choice of $\extPsiGamma{\vpt}$ on clamped and simply supported edges and the compatibility conditions on $F$.
\end{remark}

\begin{lemma}\label{lem:traceHomogeneousSpace}
	The space $\SpPsi_0$ is equal to the set of all functions $\vphit \in (H^1(\Omega))^2$ that satisfy the following boundary conditions:
	\begin{equation}\label{eq:bc_Psi0}
	\begin{aligned}
        \vphit \dotprod n & = c_E && \text{on each edge} \  E \subset \Gamma_s, \\
        \vphit & = r_C && \text{on each connected component $C$ of} \ \Gamma_f,
	\end{aligned}
	\end{equation}
	with some constants $c_E \in \mathbb{R}$ for each edge $E \subset \Gamma_s$ and functions $r_C \in RT_0$ for each connected component $C \subset \Gamma_f$ satisfying the compatibility conditions
	$c_E = r_C(x) \dotprod n_E$, if $E \subset \Gamma_s$ is an adjacent edge to a component $C$ of $\Gamma_f$, where $x$ is the enclosed corner point and $n_E$ is the normal on $E$. Moreover, for each set of constants $c_E \in \mathbb{R}$ and functions $r_C \in RT_0$ that satisfy the compatibility conditions, there is a function $\vphit \in (H^1(\Omega))^2$, for which the boundary conditions \eqref{eq:bc_Psi0} hold.
\end{lemma}
\begin{proof}
	For $\vphit \in (H^1(\Omega))^2$ and $\primalt \in C^\infty(\overline\Omega) \cap \Vprimal$ we have 
    \begin{equation}\label{eq:ibp_coupling_cond}
	\begin{aligned}
		\langle \pp{\vphit}{t} ,\grad \primalt \rangle_\Gamma & = \int_\Omega \bCurl \vphit : \grad^2 \primalt \ dx = - 
		\int_\Gamma \vphit \dotprod (\grad^2  \primalt) \, t \ ds \\
		& = -\sum_{E \subset  \Gamma_s} \int_E (\vphit \dotprod n) \ \pp{(\pp{\primalt}{n})}{t} \ ds -
		\sum_{E \subset  \Gamma_f} \int_E \vphit \dotprod \pp{(\grad \primalt)}{t} \ ds
    \end{aligned}
	\end{equation}
	using integration by parts and the boundary conditions for $\primalt$.
    Now let $\vphit \in \SpPsi_0$
	and let $E$ be an edge from $\Gamma_s \cup \Gamma_f$. 
	For each function $\varphi \in C_0^\infty(E)$, one can easily construct a function $\primalt \in C_0^\infty(\Omega\cup E)$ with $\primalt = 0$ and $\pp{\primalt}{n}=\varphi$ on $E$. 
	Since $\vphit \in \SpPsi_0$, we have $\langle \pp{\vphit}{t} ,\grad \primalt \rangle_\Gamma = 0$, which reduces to
	\[
	  \int_E (\vphit \dotprod n) \ \pp{\varphi}{t} \ ds = 0 
	  \quad \text{for all} \ \varphi \in C_0^\infty(E) 
	\]
	by using \eqref{eq:ibp_coupling_cond}.
	Hence, $\vphit\dotprod n $ is equal to a constant $c_E$ on $E$. 
	By a similar argument it follows that $\pp{(\vphit \dotprod t)}{t}$ is equal to a constant $a_E$ on each edge $E\in\Gamma_f$.
	
	Using that $\vphit\dotprod n$ is edgewise  constant on $\red{\Gamma_s \cup \Gamma_f}$ and $\pp{(\vphit \dotprod t)}{t}$ is edgewise constant on $\Gamma_f$, we obtain from (\ref{eq:ibp_coupling_cond}) by edgewise integration by parts:
	\begin{equation}\label{int-by-part}
		\langle \pp{\vphit}{t} ,\grad \primalt \rangle_\Gamma 
		 =  \red{-}\sum_{E \subset  \Gamma_s} \big( (\vphit \dotprod n) \ \pp{\primalt}{n}\big)\big|_{\partial E} \red{-} \sum_{E \subset \Gamma_f} \big( \vphit \dotprod \grad \primalt - \pp{(\vphit\dotprod t)}{t}  \,\primalt \big)\big|_{\partial E} .
	\end{equation}
	Now let $E$ and $E'$ be two adjacent edges from $\Gamma_f$ with the common corner point $x$ and let  $v \in C_0^\infty(\Omega \cup E \cup E' \cup \{x\})$. Then the condition $\langle \pp{\vphit}{t} ,\grad \primalt \rangle_\Gamma = 0$ reduces to
	\[
       0 = \vertexJump{ \vphit \dotprod \grad \primalt - \pp{(\vphit\dotprod t)}{t}  \,\primalt}_x
         = \vertexJump{ \vphit}_x \dotprod \grad \primalt(x) - \vertexJump{\pp{(\vphit\dotprod t)}{t}}_x  \,\primalt(x).
	\]
	Observe that $\vphit \in P_1$ on each edge of $\Gamma_f$ and is the trace of an $H^1$-function. Therefore, $\vphit$ must be continuous on $\Gamma_f$, which implies $\vertexJump{ \vphit}_x = 0$. Since $v(x)$ can be chosen arbitrarily, it follows that $\vertexJump{\pp{(\vphit\dotprod t)}{t}}_x = 0$. So $\pp{(\vphit\dotprod t)}{t}$ is not only constant on each edge $E$ of $\Gamma_f$ but it is equal to the same constant $a_C$ on each edge of a connected component $C$ of $\Gamma_f$.
	Since we additionally know from above that $\vphit \dotprod n$ is constant on such an edge, it easily follows that $\pp{\widetilde\vphit}{t} = 0$ on each edge of $C$ with $\widetilde\vphit(x) = \vphit(x) - a_C x$. Since $\vphit$ is continuous on $\Gamma$, $\widetilde\vphit$ is continuous, too. Then
	it follows that $\widetilde\vphit$ is equal to a common constant $b_C \in \mathbb{R}^2$ on $C$. Hence, $\vphit(x) = a_C x + b_C$ on $C$.
	
	Let $E \subset  \Gamma_s$ be an edge adjacent to a connected component $C$ of $\Gamma_f$ with enclosed corner point $x$. By a similar argument as above one can deduce from $\langle \pp{\vphit}{t} ,\grad \primalt \rangle_\Gamma = 0$ and \eqref{int-by-part} the compatibility condition $c_E = r_C(x) \dotprod n_E$.
	
	For proving the reverse direction, let $\vphit \in (H^1(\Omega))^2$ satisfy the boundary conditions \eqref{eq:bc_Psi0}. Using \eqref{int-by-part} we have
	\begin{equation*}
		\langle \pp{\vphit}{t} ,\grad \primalt \rangle_\Gamma = 
		\sum_{E \subset  \Gamma_s} \big( (\vphit \dotprod n) \ \pp{\primalt}{n}\big)\big|_{\partial E} + \sum_{E \subset \Gamma_f} \big( \vphit \dotprod \grad \primalt - \pp{(\vphit\dotprod t)}{t}  \,\primalt \big)\big|_{\partial E} = 0,
	\end{equation*}
	for all $\primalt \in C^\infty(\overline\Omega)\cap W$ for the following reasons: $\pp{\primalt}{n}(x)=0$ on all interior corner points of $\Gamma_s$, $\vphit \dotprod\grad \primalt - \pp{(\vphit\dotprod t)}{t} \primalt$ is continuous on $ \Gamma_f$, $v = 0$ and $\grad v = 0$ on corner points on the interface of $\Gamma_s$ or $\Gamma_f$ with $\Gamma_c$, and the compatibility conditions on corner points on the interface of $\Gamma_s$ and $\Gamma_f$. So it follows that $\vphit \in \SpPsi_0$.
	
	\red{Note, $\pp{\primalt}{n}(x)=0$ on all interior corner points of $\Gamma_s$, since $\primalt = 0$ on $\Gamma_s$. This makes the tangential derivative vanish at this corner points for two linear independent tangential directions. Therefore, we obtain $\grad \primalt(x) = 0$ and $\pp{\primalt}{n}(x)=0$}

	For the last part, let a set of data $c_E \in \mathbb{R}$ for each edge $E \subset  \Gamma_s$ and $r_C \in RT_0$ for each connected component $C \subset \Gamma_f$ be given, which satisfy the compatibility conditions. We will first construct a continuous function $\hat\vphit_\Gamma$ on $\Gamma$ with $\vphit_\Gamma|_E \in P_1$ for each edge $E \in \GammaEdges$, where $P_1$ is the set of polynomials  of degree $\le 1$, by prescribing its values on all corner points $x \in \GammaVertices$ as follows:
	\begin{equation*}
	   \hat\vphit_\Gamma(x) = 
	     \begin{cases}
	       r_C(x) & \quad \text{if $x$ is a corner points of the connected component} \ C \subset \Gamma_f, \\
	       \vphit_s(x) & \quad \text{if $x$ is an interior corner point of} \ \Gamma_s, \\
	       \vphit_{sc}(x) & \quad \text{if $x$ is a corner point on the interface of} \ \Gamma_s \ \text{and} \ \Gamma_c, \\
	       0 & \quad \text{if $x$ is an interior corner point of} \ \Gamma_c.
	     \end{cases}
	\end{equation*}
	Here $\vphit_s(x)$ and $\vphit_{cs}(x)$ are defined as follows: 
	For a common corner point $x$ of $E, E' \subset \Gamma_s$, $\vphit_s(x)$ is uniquely given by
	\[
	  \vphit_s(x) \dotprod n_E = c_E 
	  \quad \text{and} \quad
	  \vphit_s(x) \dotprod n_{E'} = c_{E'} .	  
	\]
	For a common corner point $x$ of $E \subset \Gamma_s$ and $E' \subset \Gamma_c$, $\vphit_{sc}(x)$ is uniquely given by
	\[
	  \vphit_{sc}(x) \dotprod n_E = c_E 
	  \quad \text{and} \quad
	  \vphit_{sc}(x) \dotprod t_E = 0 .	  
	\]
	Let  $\hat\vphit \in (H^1(\Omega))^2$ be the harmonic extension of $\hat\vphit_\Gamma$. It is easy to verify that the boundary conditions \eqref{eq:bc_Psi0} are satisfied.
\end{proof}

In order to eliminate $c_E$ and $r_C$ in \eqref{eq:bc_Psi0} we introduce the following Cl\'{e}ment-type projection operator on $\Gamma$.

\begin{definition} \label{Clement}
Let $\vphit_\Gamma \in L^2(\Gamma)$. Then the projection $\Pi_\Gamma \colon L^2(\Gamma) \rightarrow L^2(\Gamma)$ is given by $\hat\vphit_\Gamma = \Pi_\Gamma \vphit_\Gamma$, where $\hat{\vphit}_\Gamma$ is constructed as in the last part of the proof of \lemref{lem:traceHomogeneousSpace} from the data $c_E(\vphit_\Gamma)$ and $r_C(\vphit_\Gamma)$, which are the $L^2$-projection of $\vphit_\Gamma \dotprod n_E$ onto the set $P_0$ of constant functions and the $L^2$-projection of $\vphit_\Gamma|_C$ onto $RT_0$, respectively.
\red{Except if $E \subset \Gamma_s$ is an adjacent edge to a component $C$ of $\Gamma_f$, then $c_E(\vphit_\Gamma) = r_C(\vphit_\Gamma)(x) \dotprod n_E$, where $x$ is the enclosed corner point and $n_E$ is the normal on $E$, in order to enforce the compatibility conditions.}
\end{definition}

With this notation the boundary conditions \eqref{eq:bc_Psi0} can be rewritten as
	\begin{equation}\label{eq:bc_Psi0_new}
	    (\filter\vphit) \dotprod n = 0 \quad \text{on} \  \Gamma_s, \qquad
        \filter\vphit = 0 \quad \text{on} \ \Gamma_f,
	\end{equation}
with $\filter\vphit = (I - \Pi_\Gamma) \vphit$.

\subsection{Decoupled formulation}
\label{sec:decoupled_formulation}

Using the representations
\begin{equation*}
	\aux = \vp \Id + \symbCurl \vphi,\quad \auxt = \vpt \Id + \symbCurl \vphit
\end{equation*}
together with \eqref{eq:VRegDecomp} leads to the following equivalent formulation of \eqref{eq:newMixed_formulation}: 
Find $\vp \in \Q$, $\vphi \in \SpPsi_\vp = \extPsi{\vp} + \SpPsi_0$ and $\primal \in \Q$ such that
\begin{alignat*}{4}
	&(\vp \Id + \symbCurl \vphi, \vpt \Id + \symbCurl \extPsi{\vpt})_{\stiffnessC^{-1}}  &&- (\grad \primal, \grad \vpt)  &&= 0\\
	&(\vp \Id + \symbCurl \vphi, \symbCurl \vphit_0)_{\stiffnessC^{-1}}  && &&= 0\\
	&-(\grad \vp, \grad \primalt) && && = -\langle F, \primalt \rangle.
\end{alignat*}
for all $\vpt \in \Q$, $\vphit_0 \in \SpPsi_0$ and $\primalt \in \Q$. \red{For the representations of $\Opdivdiv \aux$ and $\Opdivdiv \auxt$ recall \thmref{thm:regular_decomposition}, in particular for $\aux$ we use identity \eqref{PoissonNq} and for $\auxt$ we rely on the reverse direction.}
Therefore, the mixed formulation of the Kirchhoff plate bending problem is equivalent to three (consecutively to solve) elliptic second-order problems:
\begin{enumerate}
\item
The $\vp$-problem: Find $p \in \Q$ such that
\begin{equation}\label{eq:p_problem}
	(\grad \vp, \grad \primalt) = \langle F, \primalt \rangle
	\quad \text{for all} \ \primalt \in \Q .
\end{equation}
\item
The $\vphi$-problem: For given $\vp \in \Q$, find  $\vphi \in \SpPsi_\vp = \extPsi{\vp} + \SpPsi_0$ such that
\begin{equation}\label{eq:phi_problem}
   (\symbCurl \vphi, \symbCurl \vphit_0)_{\stiffnessC^{-1}}  = 
   - (\vp \Id, \symbCurl \vphit_0)_{\stiffnessC^{-1}} 
   \quad \text{for all} \ \vphit_0 \in \SpPsi_0 .
\end{equation}

\item
The $\primal$-problem: For given $\aux = \vp \Id + \symbCurl \vphi$, find $\primal \in \Q$ such that
\begin{equation}\label{eq:w_problem}
	(\grad \primal, \grad \vpt) 
	  = (\aux, \vpt \Id + \symbCurl \extPsi{\vpt} )_{\stiffnessC^{-1}}
    \quad \text{for all} \ \vpt \in \Q.
\end{equation}
\end{enumerate}

\begin{remark}
	For the rotated function 
	$\vphiRot = (-\vphi_2,\vphi_1)^T$ 
	the $\vphi$-problem becomes a linear elasticity problem 
	\begin{equation*}
		(\varepsilon (\vphiRot), \varepsilon (\vphitRot_0))_{\hat\stiffnessC^{-1}}  = - (\vp \Id, \varepsilon (\vphitRot_0))_{\hat\stiffnessC^{-1}}
	\end{equation*}
	with a appropriately rotated material tensor $\hat \stiffnessC^{-1}$. 
\end{remark}

\begin{remark}
	If we only have clamped and simply supported boundary parts, we know from \sectref{sec:coupling_condition} that $\extPsi{\vpt} = 0$ on $\Omega$. Therefore, the terms involving $\extPsi{\vpt}$ just vanish. Furthermore, in the purely clamped case $\SpPsi_\vp$ and $\SpPsi_0$ become $(H^1(\Omega))^2$.
\end{remark}

\begin{remark}
\begin{enumerate}
\item
The $\vp$-problem and the $\primal$-problem are standard Poisson problems with mixed boundary conditions on $\Gamma_c\cup\Gamma_s$ and $\Gamma_f$.
\item
\red{
From \eqref{eq:phi_problem} we obtain
\begin{equation}\label{eq:RotM}
	\brot (\stiffnessC^{-1}\aux) = 0 \quad \text{in} \ \red{L^2(\Omega)},
\end{equation}
therefore, $\stiffnessC^{-1}\aux \in \HRot = \{\auxt \in \Ltwob: \brot \auxt \in (L^2(\Omega))^2 \}$.
With the representation $\aux = \vp\Id + \symbCurl \vphi$, using $\vp \in \Q \subset H^1(\Omega)$,} we receive that
the solution of the $\vphi$-problem satisfies the second-order differential equation
\begin{equation}\label{eq:RotsymCurl}
	\brot \left(\stiffnessC^{-1} \symbCurl \vphi\right) = - \brot \left(\stiffnessC^{-1} (\vp \Id)\right) 
	\quad \text{in} \ \red{L^2(\Omega)} 
\end{equation}
the (essential) boundary conditions (see \eqref{eq:bc_Psi0_new})
\begin{equation}\label{eq:essential_bc_phi}
        (\filter{\red{\vphi}}) \dotprod n = (\filter{\extPsiGamma{\vp}}) \dotprod n \quad \text{on} \  \Gamma_s, \qquad
        \filter {\red{\vphi}}  = \filter{\extPsiGamma{\vp}} \quad \text{on} \ \Gamma_f,
	\end{equation}
and the following condition for the flux $\chi = \left(\stiffnessC^{-1} \symbCurl \vphi\right) t$:
\begin{equation}\label{eq:chi_ibp_boundary}
  \langle \chi,\vphit\rangle_\Gamma = -\big( (\stiffnessC^{-1} p \, \Id)t,\vphit \big)_\Gamma \quad \text{for all} \ \vphit \in \SpPsi_0.
\end{equation}
If $\chi \in L^2(\Gamma)$, \red{then \eqref{eq:chi_ibp_boundary} is equivalent to} the following (natural) boundary conditions
\begin{equation}\label{eq:natural_bc_phi}
	\chi = 0 \quad \text{on} \ \Gamma_c, 
	\qquad
	\chi \cdot t = 0 \quad \text{on} \ \Gamma_s ,
\end{equation}
\red{
and
\begin{equation}\label{eq:chi_compatible}
	( \chi \dotprod n, \vphit \dotprod n )_{\Gamma_s} + ( \chi, \vphit)_{\Gamma_f} = -((\stiffnessC^{-1} \vp \Id)t, \vphit)_{\Gamma_f} \quad \text{for all} \ \vphit \in \SpPsi_0,
\end{equation}
where $p = 0$ on $\Gamma_c\cup \Gamma_s$ is used (see the $\vp$-problem). 
}

\item
For the rotated function $\vphi^\perp$, \eqref{eq:RotsymCurl} is a linear elasticity problem
\begin{equation*}
	- \bdiv \big(\hat\stiffnessC^{-1} \symbGrad (\vphiRot)\big) = \bdiv \big(\hat\stiffnessC^{-1} (\vp \Id) \big) \quad \text{in} \ \Omega
\end{equation*}
with the corresponding boundary conditions for $\vphiRot$ and $\big(\hat\stiffnessC^{-1} \symbGrad (\vphiRot)\big) n$, which can be interpreted as displacement and traction.
\end{enumerate}
\end{remark}

\begin{remark}
	\red{
	So far we have only considered homogeneous boundary conditions. In the following we indicate how to adapt the decoupled formulation introduced above to inhomogeneous boundary conditions of the form (cf. \cite{reddy_2007, harari_sokolov_krylov_2011}):
	\begin{equation*}
	\begin{alignedat}{4} 
		&\primal = \hat\primal, \quad &&\pp{\primal}{n}= \hat\theta & \quad \text{on} \ \Gamma_c, \\
		&\primal = \hat\primal, \quad &&\aux_{nn} = \hat\aux_{nn} & \quad \text{on} \ \Gamma_s, \\
		&\aux_{nn} = \hat\aux_{nn}, \quad &&\pp{\aux_{nt}}{t} + \bdiv \aux \dotprod n = \hat V_n & \quad \text{on} \ \Gamma_f,
	\end{alignedat}
	\end{equation*}
	and the corner forces
	\begin{equation*}
		\vertexJump{\aux_{nt}}_x = \hat R_x  \quad \text{for all} \ x \in \GammaVertices[f].
	\end{equation*}
	\begin{enumerate}
		\item In the $\vp$-problem \eqref{eq:p_problem} the additional contribution on the right-hand side is given as
		\begin{equation*}
			\int_{\Gamma_f} \hat V_n \primalt \ ds.
		\end{equation*}
		\item In the $\vphi$-problem \eqref{eq:phi_problem} the boundary value $\extPsiGamma{\vp}$ needed in the construction of $\extPsi{\vp}$ has to be adapted.
		With the same notations as used for  \eqref{extensionGammaE} we set
		\begin{equation*}
			\extPsiGamma{\vp}(x) = \int_0^{\sigma} (-\vp \, n + \hat\aux_{nn} \, n + c \, t) \ ds 
			\quad \text{with} \quad x = \gamma(\sigma)
			\quad \text{on} \ \Gamma  \setminus E,
		\end{equation*}
		  where  $c$ is edgewise constant, given by
		  \begin{equation*}
			c|_{E_k} = \sum_{i=1}^{k-1} \hat R_{x_i}, \quad \text{for all $k=1,\dots,K$}.
		\end{equation*}
		Here, the edges $E_k$ for $k=1,\dots,K$ are numbered consecutively in counterclockwise direction with $E_1$ starting from $x_B = \gamma(0)$. Furthermore, we denote the vertex at the end point of $\overline E_k$ by $x_k$ and use the convention $\hat R_{x_k} = 0$ for corner points $x_k \not\in \GammaVertices[f]$.
		
		As additional contribution on the right-hand side of the $\vphi$-problem \eqref{eq:phi_problem} we obtain
		\begin{equation*}
			-\int_{\Gamma_c} \pp{\vphit}{t} \dotprod n \ \hat\theta + \pp{\vphit}{t} \dotprod t \ \pp{\hat\primal}{t} \ ds - \int_{\Gamma_s} \pp{\vphit}{t} \dotprod t \ \pp{\hat\primal}{t} \ ds,
		\end{equation*}
		provided $\pp{\vphit}{t} \in (L^2(\Omega))^2$.
		\item In the $\primal$-problem \eqref{eq:w_problem} we get as additional contribution $-(\grad \overline \primal, \grad \vpt)$ on the right-hand side, where $\overline \primal \in H^1(\Omega)$ is any extension of the Dirichlet data $\hat\primal$.
	\end{enumerate}
	}
\end{remark}


\section{The discretization method}
\label{sec:conforming_method}

Let $(\decomposition)_{h \in \mathcal{H}}$ be a \red{shape-regular} family of subdivision\red{s} of the domain $\Omega$ into polygonal elements. 
The diameter of an element $T \in \decomposition$ is denoted by $h_T$ and we define $h= \max \{ h_T : T \in \decomposition \}$.
We denote the set of all edges $e$ of elements $T \in \decomposition$ with $e \subset \Gamma_s$ by $\edges[s]$ and with $e \subset \Gamma_f$ by $\edges[f]$. 
The length of an edge $e$ is denoted by $h_e$.
Moreover, we assume that the total number of edges in $\edges[s] \cup \edges[f]$ is bounded by $c \, h^{-1}$.
Here and in the sequel $c$ denotes a generic constant independent of $h$, possibly different at each occurrence. \red{Note, this assumption is weaker than requiring a quasiuniform family of subdivisions. It can be viewed as a quasiuniformity condition in a neighborhood of the boundary.}

\begin{remark}
Observe that the symbol $e$ is used to indicate element edges, while the symbol $E$ used in the preceding sections is reserved for edges of the domain as introduced at the beginning of Section \ref{sec:kirchhoff_love_plate}.
\end{remark}

Let $\FE_h$ be a finite dimensional subspace of $H^1(\Omega)$ of piecewise polynomials with degree $k$ associated with $\decomposition$ and we set $\FE_{h,0} = \FE_h \cap H^1_{0,\Gamma_c \cup \Gamma_s}(\Omega)$.
For $\primalt \in H^s(\Omega)$, with $1\leq s \leq k+1$, we assume the standard approximation property
\begin{equation}\label{eq:approximation_property}
	\tag{A1}
	\inf_{\primalt_h \in \FE_h} \|\primalt - \primalt_h \|_l \leq c \ h^{s-l} \|\primalt\|_s, 
\end{equation}
for all $l \in \{0, \dots, s\}$. Moreover, we require the discrete trace inequality
\begin{equation}\label{eq:discrete_trace_inequality}
	\tag{A2}
	\|\primalt_h\|_{0,e} \leq c \, h_T^{-\frac{1}{2}} \|\primalt_h\|_{0,T} \quad \text{for all} \ \primalt_h \in \FE_h, \ e \subset \partial T, \ T \in \decomposition,
\end{equation}
and the continuous trace inequality
\begin{equation}\label{eq:continuous_trace_inequality}
	\tag{A3}
	\|\primalt\|_{0,e} \leq c \, h_T^{-\frac{1}{2}} (\|\primalt\|_{0,T} + h_T \|\grad \primalt \|_{0,T}) \quad \text{for all} \ \primalt \in H^1(T), \ e\subset\partial T, \ T \in \decomposition.
\end{equation}
Here and in the following the $L^2$-norm on an element $T$ and an edge $e$ are denoted by $\|.\|_{0,T}$ and $\|.\|_{0,e}$, respectively.
The properties (A1), (A2), (A3) are satisfied for standard finite element spaces or isogeometric B-spline discretization spaces under standard assumptions. 

The method we propose 
consists of three consecutive steps:
\begin{enumerate}
	\item The discrete $\vp$-problem: Find $\vp_h \in \FE_{h,0}$ such that
	\begin{equation}\label{eq:discrete_p_problem}
		(\grad \vp_h, \grad \primalt_h) = \langle F, \primalt_h \rangle 
	    \quad \text{for all} \ \primalt_h \in \FE_{h,0}.
	\end{equation}
	\item The discrete $\vphi$-problem: 
	For given $\vp_h \in \FE_{h,0}$, find $\vphi_h \in (\FE_h)^2/RT_0$ such that
	\begin{equation}\label{eq:discrete_phi_problem}
		a_{\vphi,h}(\vphi_h, \vphit_h) = \langle F_{\vphi,h}, \vphit_h \rangle
        	\quad \text{for all} \ \vphit_h \in (\FE_h)^2/RT_0, 
   	\end{equation}
	with
	\begin{align*}
		a_{\vphi,h}(\vphi, \vphit) &= (\symbCurl \vphi, \symbCurl \vphit)_{\stiffnessC^{-1}}+ s(\vphi, \vphit)
		 + s(\vphit, \vphi)  + r_h(\vphi, \vphit), \\
		\langle F_{\vphi,h}, \vphit \rangle &= -(\vp_h \Id, \symbCurl \vphit)_{\stiffnessC^{-1}}
		  - c(\vp_h, \vphit)  + s(\vphit, \extPsiGamma{\vp_h}) + r_h(\extPsiGamma{\vp_h}, \vphit),
        \end{align*}
	where
	\begin{align*}
		s(\vphi, \vphit) 
		& = \left(  \chi \dotprod n , \filter\vphit \dotprod n \right)_{\Gamma_s} 
		  + \left( \chi , \filter\vphit \right)_{\Gamma_f}
		  \quad \text{with} \quad \chi = (\stiffnessC^{-1} \symbCurl \vphi)  t, \\
		c(\vpt,\vphit) & = \big((\stiffnessC^{-1} \vpt\Id) t, \filter\vphit \big)_\Gamma,\\
		r_h(\vphi, \vphit) & = \sum_{e\in \edges[s]} \frac{\eta}{h_e} 
		\left( \filter\vphi \dotprod n ,  \filter\vphit \dotprod n \right)_e + \sum_{e\in\edges[f]} \frac{\eta}{h_e} \left( \filter\vphi , \filter\vphit \right)_e,
	\end{align*}
	for some penalty parameter $\eta > 0$. 
	\item The discrete $\primal$-problem:
	For given $\aux_h = \vp_h \Id + \symbCurl \vphi_h$, find $\primal_h \in \FE_{h,0}$ such that
	\begin{equation}\label{eq:discrete_w_problem}
		(\grad \primal_h, \grad \vpt_h) = \langle F_{\primal,h}, \vpt_h \rangle
		\quad \text{for all} \ \vpt_h \in \FE_{h,0}, 
	\end{equation}
	with
	\begin{equation}\label{eq:rhs_discrete_w_problem}
	\begin{aligned}
		\langle F_{\primal,h}, \vpt \rangle 
		&= (\aux_h, \vpt \Id)_{\stiffnessC^{-1}} 
		  - s(\vphi_h, \extPsiGamma{\vpt}) - c(\vp_h, \extPsiGamma{\vpt})\\
		  & \quad {} - r_h(\vphi_h- \extPsiGamma{\vp_h}, \extPsiGamma{\vpt}).
	\end{aligned}
	\end{equation}
\end{enumerate}

\begin{remark}\label{rem:consistency}
\begin{enumerate}
\item
The discrete $\vp$-problem is the standard Galerkin method applied to \eqref{eq:p_problem}.
\item
The discrete $\vphi$-problem is a Nitsche method applied to \eqref{eq:phi_problem}, which is derived as follows.
We start from the identity
	\begin{equation}\label{eq:Nitsche_identity} 
	\begin{aligned}
		&(\aux, \symbCurl \vphit)_{\stiffnessC^{-1}}
		+ \langle (\stiffnessC^{-1}\aux) t, \vphit \rangle_\Gamma = 0
		\quad \text{for all} \ \vphit \in (H^1(\Omega))^2,
	\end{aligned}
	\end{equation}
which follows \red{from (\ref{eq:RotM}) by multiplying with a test function and} using integration by parts.
\red{
Plugging in the second term of \eqref{eq:Nitsche_identity} the representation of $\aux$ leads to}
\begin{equation}
	\langle (\stiffnessC^{-1}\aux) t, \vphit \rangle_\Gamma = \langle \chi, \vphit \rangle_\Gamma + \big((\stiffnessC^{-1} \vp\Id) t, \vphit \big)_\Gamma
\end{equation}
with $\chi = (\stiffnessC^{-1} \symbCurl \vphi) t$.
\red{For the next step note that $\vphit - \filter \vphit = \Pi_\Gamma \vphit$ on $\Gamma$ and according to \lemref{lem:traceHomogeneousSpace} there exists a $\vphit_0 \in \SpPsi_0$ such that $\vphit_0 = \Pi_\Gamma \vphit$ on $\Gamma$. Then the natural boundary conditions \eqref{eq:natural_bc_phi} and \eqref{eq:chi_compatible} (with $\vphit$ replaced by $\vphit_0$) lead to}
\begin{equation}\label{eq:Nitsche_identity_bc}
\begin{aligned}
	\langle (\stiffnessC^{-1}\aux) t, \vphit \rangle_\Gamma &= \left(  \chi \dotprod n , \filter\vphit \dotprod n \right)_{\Gamma_s} + \left( \chi , \filter\vphit \right)_{\Gamma_f} + \big((\stiffnessC^{-1} \vp\Id) t, \filter\vphit \big)_\Gamma\\
	&= s(\vphi,\vphit) + c(\vp, \vphit),
\end{aligned}
\end{equation}
provided $\chi \in L^2(\Gamma)$.
Then the method is obtained by first extending \eqref{eq:Nitsche_identity} by the terms $s(\vphit,\vphi - \extPsiGamma{\vp})$ and $r_h(\vphi - \extPsiGamma{\vp},\vphit)$, which vanish for the exact solution (see \eqref{eq:essential_bc_phi}), and then by replacing $\vp$, $\vphi$, and $\vphit$ with $\vp_h$, $\vphi_h$, and $\vphit_h$.   
Following \cite{diPietro_ern_2012} we call $s(\vphi,\vphit) + c(\vp,\vphit)$,
$s(\vphit,\vphi)$, and $r_h(\vphi,\vphit)$ the consistency, symmetry and penalty terms, respectively.
\item
The discrete $\primal$-problem is the standard Galerkin method applied to \eqref{eq:w_problem}, where the right-hand side is reformulated as follows. By using \eqref{eq:Nitsche_identity} and \eqref{eq:Nitsche_identity_bc}  with $\vphit = \extPsi{\vpt}$ we obtain for the right-hand side of (\ref{eq:w_problem}):
\begin{equation}\label{eq:rhs_w_problem}
\begin{aligned}
  & (\aux, \vpt \Id + \symbCurl \extPsi{\vpt} )_{\stiffnessC^{-1}} \\
  & \quad = (\aux, \vpt \Id)_{\stiffnessC^{-1}} 
       - s(\vphi, {\extPsiGamma{\vpt}})- c(\vp, {\extPsiGamma{\vpt}}) - r_h(\vphi - \extPsiGamma{\vp},\extPsiGamma{\vpt}),
\end{aligned}
\end{equation}
\red{where we additionally extend by the term $r_h(\vphi - \extPsiGamma{\vp}, \extPsiGamma{\vpt})$, which vanishes for the exact solution.} Then \eqref{eq:rhs_discrete_w_problem} is obtained by replacing $\vp$, $\vphi$, and $\vpt$ with $\vp_h$, $\vphi_h$, and $\vpt_h$.
\end{enumerate}
\end{remark}

\begin{remark}
	Since only $\extPsiGamma{\vp_h}$ and $\extPsiGamma{\vpt_h}$ appear in the numerical method, the extensions of $\extPsi{\vp_h}$ and $\extPsi{\vpt_h}$ to the interior are not needed.
	Although the functions $\extPsiGamma{\vpt_h}$ do not have local support, the linear systems can still be assembled with optimal complexity.
\end{remark}

\begin{remark}
	Our decomposition of the continuous problem also leads to a new interpretation of the well-known Hellan-Herrmann-Johnson (HHJ) method; see \cite{hellan:67,herrmann:67,johnson:73}. 
	We can proceed similar as in \thmref{thm:regular_decomposition} and derive a discrete regular decomposition for the approximation space of the auxiliary variable, leading to a reformulation of the HHJ method in form of three consecutively to solve discretized second-order problems, as it has already been worked out in details in the purely clamped situation in \cite{krendl_rafetseder_zulehner_2016}.
	\red{The HHJ method is mainly restricted to triangular meshes. In \cite{scapolla_1980} a HHJ-type method on rectangular meshes is considered.} The new method introduced above is more flexible in the sense that triangular and \red{general} quadrilateral meshes can be handled and also
	isogeometric B-spline discretization spaces can be used.
\end{remark}
The main result of this section is the following a priori discretization error estimate for the proposed method.
\begin{theorem}
\label{thm:error_estimate}
	Assume that $\Omega$ is convex. Let $(\primal, \aux)$, with $\aux = \vp \Id + \symbCurl \vphi$, be the solution of the mixed formulation $\eqref{eq:newMixed_formulation}$ and  $(\primal_h, \aux_h)$, with $\aux_h = \vp_h \Id + \symbCurl \vphi_h$, be the approximate solution, given by \eqref{eq:discrete_p_problem}, \eqref{eq:discrete_phi_problem}, \eqref{eq:discrete_w_problem}. 
	For $\primal \in H^{s_\primal}(\Omega)$, $\vp \in H^{s_\vp}(\Omega)$ and $\vphi \in (H^{s_\vphi}(\Omega))^2$, with $1\leq s_\primal, s_\vp \leq k+1$ and $\frac{3}{2} < s_\vphi \leq k+1$, we have the estimate
	\begin{equation*}
		\| \aux - \aux_h\|_0 + \| \primal - \primal_h \|_1 \leq c \left( h^{s_\primal-1} \|\primal\|_{s_\primal} + h^{s_\vp-1} \|\vp\|_{s_{\vp}} + h ^{s_\vphi-1} \|\vphi\|_{s_{\vphi}} \right).
	\end{equation*}
	Especially, for $s_\vp = k+1$, $s_\phi = k+1$ and $s_\primal = k+1$ we obtain
	\begin{equation*}
		\| \aux - \aux_h\|_0 + \| \primal - \primal_h \|_1 \leq c \ h^k \left( \|\primal\|_{k+1} + \|\vp\|_{k+1} + \|\vphi\|_{k+1} \right).
	\end{equation*}
\end{theorem}

We will derive these estimates by discussing the discretization errors of the $\vp$-problem, the $\vphi$-problem, and the $\primal$-problem consecutively.

\subsection{Error estimates for the \texorpdfstring{$\vp$}{p}-problem}

We start with the $\vp$-problem (\ref{eq:p_problem}) and its discretization (\ref{eq:discrete_p_problem}). It is well-known that 
the following error estimates hold:
\begin{lemma}[$\vp$-problem]\label{lem:p_error_estimate}
	Under the assumptions of \thmref{thm:error_estimate} we have
	\begin{equation*} 
			\| \vp - \vp_h \|_1 \leq c \ h^{s_\vp-1} \|\vp\|_{s_\vp}, \quad \| \vp - \vp_h \|_0  \leq c \ h^{s_\vp} \|\vp\|_{s_\vp},\quad 
		\|\vp - \vp_h\|_{0,\Gamma} 
		  \leq c \ h^{s_\vp-\frac{1}{2}} \|\vp\|_{s_\vp}.
	\end{equation*}
\end{lemma}
We refer to standard literature for the proof. 
Note that the convexity of $\Omega$ is used only for the $L^2$-estimates.

\subsection{Error estimates for the \texorpdfstring{$\vphi$}{phi}-problem}

We follow the standard approach as outlined, e.g.,  in \cite{diPietro_ern_2012} and introduce two mesh-dependent semi-norms: the jump semi-norm $|\vphit|_{\frac{1}{2},h}$ and the average semi-norm $|\chi|_{-\frac{1}{2},h}$, which here are given by
\begin{align*}
   |\vphit|_{\frac{1}{2},h}^2 
     &= \sum_{e\in \edges[s]} h_e^{-1} \, \|\vphit \dotprod n\|^2_{0,e}
        + \sum_{e\in \edges[f]} h_e^{-1} \, \|\vphit\|^2_{0,e}, \\
   |\chi|_{ -\frac{1}{2},h}^2
     &= \sum_{e\in \edges[s]} h_e \, \|\chi \dotprod n\|^2_{0,e}
        + \sum_{e\in \edges[f]} h_e \, \|\chi\|^2_{0,e}.
\end{align*}
 
The analysis relies on the discrete coercivity and the boundedness of the bilinear form $a_{\vphi,h}$ in appropriate norms. \begin{lemma}[Coercivity]\label{coercivity}
	There is a constant $c> 0$ such that 
	\begin{equation*}
		a_{\vphi,h}(\vphit_h, \vphit_h) \geq c \, \|\vphit_h\|^2_h \quad \text{for all} \ \vphit_h \in (\FE_h)^2 /RT_0 ,
	\end{equation*}
 	provided $\eta$ is sufficiently large, where the mesh-dependent norm $\|\vphi\|_h$ is given by
\begin{equation*}
	\|\vphi\|_h^2 = (\symbCurl \vphi,\symbCurl \vphi)_{\stiffnessC^{-1}} + 
	|\filter\vphi|^2_{\frac{1}{2},h} .
\end{equation*}
\end{lemma}

\begin{lemma}[Boundedness]
	There is a constant $c > 0$ such that
	\begin{equation*}
		a_{\vphi,h}(\vphi, \vphit_h) \leq c \, \|\vphi\|_{h,*} \|\vphit_h\|_h
	\end{equation*}
	for all $\vphi \in (H^s(\Omega))^2 + (\FE_h)^2$, with $s>\frac{3}{2}$ and $\vphit_h \in (\FE_h)^2$,
	where the mesh-dependent norm $\| \vphi\|_{h,*}$ is given by 
\begin{align*}
	\| \vphi\|_{h,*}^2 &= \|\vphi\|^2_h + |\chi|^2_{-\frac{1}{2},h} \quad \text{with} \quad \chi = \left(\stiffnessC^{-1} \symbCurl \vphi\right) t.
\end{align*}
\end{lemma}
The proofs of these two lemmas are analogous to the proofs of similar results in \cite{diPietro_ern_2012} and are, therefore, omitted. However, for later use, we 
explicitly mention here the fundamental estimates for the consistency, symmetry, and penalty terms which are used for proving coercivity and boundedness: 
For all $\vphit \in (H^1(\Omega))^2$, $\xi \in (L^2(\Gamma))^2$, $q \in L^2(\Gamma)$ we have
\begin{align}
    |s(\vphit,\xi)|
    & \le c \, (\inf_{\vphit_h \in (\FE_h)^2} \|\vphit-\vphit_h\|_{h,*} + \|\vphit\|_h)
           \, |\filter\xi|_{\frac{1}{2},h} \label{eq:s_estimate} \\
    |c(\vpt,\vphit)|
    & 
      \le |(\stiffnessC^{-1} \vpt\Id)t|_{-\frac{1}{2},h} \, \|\vphit\|_{h}, \label{eq:c_estimate}
     \\
    |r_h(\xi,\vphit)| &
      \le c \, |\filter\xi|_{\frac{1}{2},h} \, \|\vphit\|_{h} \label{eq:r_estimate}
\end{align}
Obviously, \eqref{eq:s_estimate} simplifies to
\begin{equation}\label{eq:s_estimate_h}
    |s(\vphit_h,\xi)| \le c \, \|\vphit_h\|_{h} \, |\filter\xi|_{\frac{1}{2},h} 
    \quad \text{for all} \ \vphit_h \in (\FE_h)^2.\\
\end{equation}

Additionally we need an estimate of the consistency error.
\begin{lemma}[Consistency error]\label{lem:cons_error}
Under the assumptions of \thmref{thm:error_estimate} we have
\[
  \sup_{0 \neq \vphit_h \in (\FE_h)^2} \frac{|a_{\vphi,h}(\vphi,\vphit_h) - \langle F_{\vphi,h},\vphit_h\rangle|}{\|\vphit_h\|_h} \le c \, h^{s_\vp-1} \|\vp\|_{s_\vp}.
\]
\end{lemma}
\begin{proof}
\red{From \remref{rem:consistency} it follows that $\vphi$ satisfies
\begin{equation*}
	a_{\vphi,h}(\vphi, \vphit_h) = -(\vp \Id, \symbCurl \vphit_h)_{\stiffnessC^{-1}} - c(\vp, \vphit_h)  + s(\vphit_h, \extPsiGamma{\vp}) + r_h(\extPsiGamma{\vp}, \vphit_h)
\end{equation*}
for all $\vphit_h \in (\FE_h)^2/RT_0$.
Therefore, the difference to the right-hand side in \eqref{eq:discrete_phi_problem} is given as}
\begin{equation*}
	\begin{aligned}
        a_{\vphi,h}(\vphi,\vphit_h) - \langle F_{\vphi,h},\vphit_h\rangle  
        & = - \left((\vp - \vp_h) \Id, \symbCurl \vphit_h\right)_{\stiffnessC^{-1}}
		 - c(\vp-\vp_h,\vphit_h) \\
		&\quad {} + s(\vphit_h, \extPsiGamma{\vp-\vp_h}) + r_h(\extPsiGamma{\vp-\vp_h}, \vphit_h).
	\end{aligned}
\end{equation*}
From the Cauchy inequality on $\Omega$ and \eqref{eq:c_estimate}, \eqref{eq:s_estimate_h}, \eqref{eq:r_estimate} we obtain
\begin{align*}
  & |a_{\vphi,h}(\vphi,\vphit_h) - \langle F_{\vphi,h},\vphit_h\rangle| \\
  & \quad \le c \left( \|p-p_h\|_0 + |(\stiffnessC^{-1} (\vp - \vp_h)\Id)t|_{-\frac{1}{2},h}
    + |\filter{\extPsiGamma{\vp-\vp_h}}|_{\frac{1}{2},h} \right) \|\vphit_h\|_{h} .
\end{align*}
	From the continuous trace inequality \eqref{eq:continuous_trace_inequality} it follows that
	\begin{equation*}
	\begin{aligned}
		|(\stiffnessC^{-1} (\vp - \vp_h)\Id)t|_{-\frac{1}{2},h}^2 
		& \le c  \sum_{e\in \edges[f]} h_e \, \|\vp - \vp_h \|^2_{0,e} 
		\leq c \ ( \|\vp - \vp_h\|^2_0 + h^2 \|\grad(\vp - \vp_h)\|^2_0).
	\end{aligned}
	\end{equation*}
From the definition of $\Pi_\Gamma \xi$ and $\extPsiGamma{\vpt}$ one obtains
\begin{equation*}
  h_e^{-1} \|\Pi_\Gamma \xi \|_{0,e}^2 \le c \, h \, |\xi|_{\frac{1}{2},h}^2
  \quad \text{and} \quad
  h_e^{-1} \|\extPsiGamma{\vpt}\|_{0,e}^2 \le c \, \|\vpt\|_{0,\Gamma}^2
\end{equation*}
for all $e \in \edges[s] \cup \edges[f]$, $\xi \in (L^2(\Gamma))^2$, and $\vpt \in L^2(\Gamma)$. 
\red{
For the first inequality note that
\begin{equation*}
	\|\Pi_\Gamma \xi \|_{L^\infty(\Gamma)} \leq \max(|c_E(\xi)|,|a_C(\xi)|, |b_C(\xi)|),
\end{equation*}
where $c_E(\xi) \in \mathbb{R}$ and $r_C(\xi)(x) = a_C(\xi) + b_C(\xi) \, x$ with $a_C(\xi)\in\mathbb{R}^2$ and $b_C(\xi)\in\mathbb{R}$ are the data used in the construction of $\Pi_\Gamma \xi$ in \defref{Clement}. 
On an edge $E \subset \Gamma_s$ we have
\begin{align*}
	&|c_E(\xi)| = \left| \frac{1}{|E|} \int_E \xi \cdot n \ ds \right|
	\leq c \ \sum_{e \subset E} \int_e |\xi \cdot n| \ d s 
	\leq c \ \sum_{e \subset E} h_e^{1/2} \ \|\xi \cdot n\|_{0,e} \\
	& = c \sum_{e \subset E} h_e \ (h_e^{-1/2} \|\xi \cdot n\|_{0,e})
	\leq c \ h^{1/2} \, |\xi|_{\frac{1}{2},h}.
\end{align*}
By similar arguments analogous results hold for $|a_C(\xi)|$ and $|b_C(\xi)|$. Combining these estimates with
\begin{equation*}
	\|\Pi_\Gamma \xi \|_{0,e}^2 \leq h_e \|\Pi_\Gamma \xi \|^2_{L^\infty(\Gamma)}
\end{equation*}
provides the first inequality. The second inequality holds since
\begin{equation*}
	\|\extPsiGamma{\vpt}\|_{0,e}^2 = \int_e \left|\int_0^\sigma \vpt n \ ds \right|^2 d\sigma \leq \int_e \left( \int_\Gamma |\vpt| \ ds \right)^2 d\sigma \leq c \ h_e \|\vpt\|^2_{0,\Gamma}.
\end{equation*}
}

Using that the total number of edges in $\edges[s] \cup \edges[f]$ is bounded by $c \, h^{-1}$ it follows that
\begin{equation*}
  |\Pi_\Gamma \xi|_{\frac{1}{2},h} \le c \, |\xi|_{\frac{1}{2},h}
  \quad \text{and} \quad
 |\extPsiGamma{\vpt}|_{\frac{1}{2},h} \le c \, h^{-\frac{1}{2}} \, \|q\|_{0,\Gamma},
\end{equation*}
and, therefore,
    \begin{equation*}
      |\filter{\extPsiGamma{\vp-\vp_h}}|_{\frac{1}{2},h}^2
      \le c \, h^{-1} \|\vp-\vp_h\|^2_{0,\Gamma} .
    \end{equation*}
Then the estimate immediately follows from the error estimates in \lemref{lem:p_error_estimate}.
\end{proof}

From the last three lemmas we obtain the following error estimate for the $\vphi$-problem:

\begin{lemma}[$\vphi$-problem]\label{lem:phi_error_estimate}
	Under the assumptions of \thmref{thm:error_estimate} we have
	\begin{equation*}
		\| \vphi - \vphi_h \|_h \leq c \ (h^{s_\vphi-1} \|\vphi\|_{s_\vphi}+ h^{s_\vp-1} \|\vp\|_{s_\vp}).
	\end{equation*}
\end{lemma}

\begin{proof}
	From coercivity and boundedness we obtain by standard arguments: 
	\begin{equation*}
	  \|\vphi - \vphi_h\|_h \le c \, \left( \inf_{\vphit_h \in (\FE_{h})^2} \|\vphit_h - \vphi\|_{h,*} + \sup_{0 \neq \vphitt_h \in (\FE_h)^2} \frac{|a_{\vphi,h}(\vphi,\vphitt_h) - \langle F_{\vphi,h},\vphitt_h\rangle|}{\|\vphitt_h\|_h} \right).
	\end{equation*}
Since $\vphi \in (H^s(\Omega))^2$, with $\frac{3}{2}<s \leq k+1$, it follows from
assumptions \eqref{eq:approximation_property} and \eqref{eq:continuous_trace_inequality}
that
\begin{equation}\label{eq:approximation_property_h*}
	\tag{A3*}
	\inf_{\vphit_h \in (\FE_h)^2} \|\vphi - \vphit_h\|_{h,*} 
	  \leq c \ h^{s-1} \|\vphi\|_s.
\end{equation}
This approximation property and \lemref{lem:cons_error} directly imply the error estimate.
\end{proof}

\subsection{Error estimates for the \texorpdfstring{$\primal$}{w}-problem}

\begin{lemma}[$\primal$-problem] \label{lem:w_error_estimate}
	Under the assumptions of \thmref{thm:error_estimate} we have
	\begin{equation*}
		\| \primal - \primal_h \|_1 \leq c \left( h^{s_\primal-1} \|\primal\|_{s_\primal} + h^{s_\vp-1} \|\vp\|_{s_\vp} + h^{s_\vphi-1} \|\vphi\|_{s_\vphi} \right).
	\end{equation*}
\end{lemma}
\begin{proof}
	The first lemma of Strang provides
	\begin{equation*}
		\| \primal - \primal_h \|_1 \leq c \ \left( \inf_{\primalt_h \in \FE_{h,0}} \| \primal - \primalt_h \|_1 + \sup_{ 0 \neq \vpt_h \in \FE_{h,0} } \frac{| (\grad \primal, \grad \vpt_h) - \langle F_{\primal,h},\vpt_h \rangle|}{\|\vpt_h\|_1} \right).
	\end{equation*}
	The first term can be estimated by the approximation property \eqref{eq:approximation_property}. 
	It remains to estimate the consistency error. 
	Using \eqref{eq:w_problem} with \eqref{eq:rhs_w_problem} and \eqref{eq:rhs_discrete_w_problem} we have
	\begin{align*}
		(\grad \primal, \grad \vpt_h) - \langle F_{\primal,h}, \vpt_h \rangle
		& = (\aux - \aux_h, \vpt_h \Id)_{\stiffnessC^{-1}} 
		     - s(\vphi-\vphi_h,\extPsiGamma{\vpt_h}) - c(\vp - \vp_h, \extPsiGamma{\vpt_h})
		  \\
		&\quad {} 
		+ r_h\bigl( \vphi_h - \extPsiGamma{\vp_h}, \extPsiGamma{\vpt_h}\bigr).
	\end{align*}
	Subtracting \eqref{eq:Nitsche_identity} and \eqref{eq:discrete_phi_problem} leads to \begin{align*}
        & (\aux-\aux_h, \symbCurl \vphit_h)_{\stiffnessC^{-1}} + s(\vphi-\vphi_h,\red{\vphit_h}) + c(\vp - \vp_h, \red{\vphit_h}) \\
		&\quad {} - s(\vphit_h, \vphi_h - \extPsiGamma{\vp_h}) - r_h(\vphi_h - \extPsiGamma{\vp_h}, \vphit_h) = 0
		\quad \text{for all} \ \vphit_h \in (\FE_h)^2.
	\end{align*}
	By subtracting the last two equations we obtain
	\begin{align*}
		&(\grad \primal, \grad \vpt_h) - \langle F_{\primal,h}, \vpt_h \rangle
		= (\aux - \aux_h, \vpt_h \Id \red{-} \symbCurl \vphit_h)_{\stiffnessC^{-1}}\\
		&\quad {} - s(\vphi-\vphi_h,\extPsiGamma{\vpt_h}-\vphit_h) - c(\vp-\vp_h,\extPsiGamma{\vpt_h}-\vphit_h)\\
		&\quad {} + s(\vphit_h, \vphi_h - \extPsiGamma{\vp_h})
	     + r_h\bigl( \vphi_h - \extPsiGamma{\vp_h}, \extPsiGamma{\vpt_h} - \vphit_h\bigr).
	\end{align*}
    The five terms on the right-hand side, denoted by $T_1$, $T_2$, \ldots $T_5$ in consecutive order of their appearance, are estimated as follows: From the Cauchy inequality on $\Omega$ and \eqref{eq:s_estimate}, \eqref{eq:c_estimate}, \eqref{eq:s_estimate_h}, \eqref{eq:r_estimate} we obtain
    \begin{align*}
      |T_1| & \le c \, (\|\vp-\vp_h\|_0 + \|\vphi-\vphi_h\|_h)  \, (\|\vpt_h\|_0 + \|\vphit_h\|_h), \\
      |T_2| & \le c \, (\inf_{\vphitt_h \in (\FE_h)^2} \|\vphi-\vphitt_h\|_{h,*} + \|\vphi-\vphi_h\|_h) \, \| \extPsiGamma{\vpt_h} - \vphit_h\|_h ,\\
      |T_3| & \le c \, |(\stiffnessC^{-1} (\vp-\vp_h)\Id)t|_{-\frac{1}{2},h} \|\extPsi{\vpt_h} - \vphit_h\|_h, \\
      |T_4| & \le c \,  |\filter(\vphi_h - \extPsiGamma{\vp_h})|_{\frac{1}{2},h} \, \|\vphit_h\|_h, \\
      |T_5| & \le c \,  |\filter(\vphi_h - \extPsiGamma{\vp_h})|_{\frac{1}{2},h} \, \|\extPsiGamma{\vpt_h} - \vphit_h\|_h
    \end{align*}
    for all $\vphit_h \in (\FE_h)^2$. In particular, we choose $\vphit_h = \Pi_h (\extPsi{\vpt_h})$, where $\Pi_h$ denotes the $L^2$-orthogonal projection onto $(\FE_h)^2$.
    Then, following , e.g., \cite{diPietro_ern_2012}, it can be shown that
	\begin{equation*}
		\|\extPsi{\vpt_h} - \vphit_h\|_h \leq c \, \ \|\extPsi{\vpt_h}\|_1 
		\quad \text{and, therefore,} \quad
		\| \vphit_h\|_h \leq c \, \ \|\extPsi{\vpt_h}\|_1.
	\end{equation*}
	With the stability estimate \eqref{eq:extension_stability} it follows that
	\begin{equation*}
		\|\extPsi{\vpt_h} - \vphit_h\|_h \leq c \, \ \|\vpt_h\|_1 \quad \text{and} \quad
		\| \vphit_h\|_h \leq c \, \ \|\vpt_h\|_1.
	\end{equation*}
	Observe that
	\begin{equation*}
	   |(\stiffnessC^{-1} (\vp-\vp_h)\Id)t|_{-\frac{1}{2},h} \le c(\|\vp-\vp_h\|_0 + h \|\grad (\vp-\vp_h)\|_0)
	\end{equation*}
	and
	\begin{align*}
	  & |\filter(\vphi_h - \extPsiGamma{\vp_h})|_{\frac{1}{2},h}
	    = 
	  |\filter(\vphi_h - \vphi) + \filter (\extPsiGamma{\vp-\vp_h})|_{\frac{1}{2},h} \\
	  & \quad \le \|\vphi_h - \vphi\|_h + |\filter {\extPsiGamma{\vp-\vp_h}}|_{\frac{1}{2},h}
	    \le \|\vphi_h - \vphi\|_h + c h^{-\frac{1}{2}} \|\vp-\vp_h\|_{0,\Gamma},
	\end{align*}
	see the estimates in the proof of \lemref{lem:phi_error_estimate}.
	Then the result follows directly from the estimates in \lemref{lem:p_error_estimate}, \lemref{lem:phi_error_estimate}, and \eqref{eq:approximation_property_h*}.
\end{proof}

Finally, we obtain the proof of the main result:
\begin{proof}[Proof of \thmref{thm:error_estimate}]
	By combining the results of \lemref{lem:p_error_estimate} and \lemref{lem:phi_error_estimate} we obtain
	\begin{equation*}
		\| \aux - \aux_h\|_0 \leq \|\vp - \vp_h \|_0 + \|\symbCurl (\vphi -\vphi_h) \|_0 \leq c \ ( h^{s_{\vp}-1} \|\vp\|_{s_\vp} + h^{s_\vphi -1} \|\vphi\|_{s_\vphi} ).
	\end{equation*}
	Together with \lemref{lem:w_error_estimate} this completes the proof.
\end{proof}

\section{Numerical experiments}
\label{sec:numerical_experiments}

We consider a square plate $\Omega = (-1,1)^2$ with simply supported north and south boundary, clamped west boundary and free east boundary. The material tensor $\stiffnessC$ is the identity, and the load is given by
\begin{equation*}
	 f(x,y) = 4 \pi ^4 \sin (\pi  x) \sin (\pi  y).
\end{equation*}
The exact solution is of the form
\begin{equation*}
	\primal(x,y) = \bigl((a + b x)\cosh(\pi x) + (c + d x)\sinh(\pi x) + \sin(\pi x)\bigr) \sin(\pi y),
\end{equation*}
which satisfies the boundary conditions on the simply supported boundary parts anyway. The constants $a, b, c$ and $d$ are chosen such that the four remaining boundary conditions (on the clamped and free boundary parts) are fulfilled, for details, see \cite{reddy_2007}.

In order to illustrate the flexibility of our discretization method we use 
for $\FE_h$  isogeometric B-spline discretization spaces of degree $k$ with maximum smoothness; see, e.g, \cite{cotrell_hughes_brazilevs_2009,daVeiga_buffa_sangalli_vazquez_2014} for information on isogeometric analysis. For $k=1$, the discretization space $\FE_h$ coincides with the standard finite element space of continuous and piecewise bilinear elements. 
In all experiments a sparse direct solver is used for each of the three sub-problems. 
The implementation is done in the framework of G+Smo
("Geometry + Simulation Modules"), an object-oriented C++ library, see \url{https://ricamsvn.ricam.oeaw.ac.at/trac/gismo/wiki/WikiStart}.

In Tables \ref{tab:mixed_k_1}, \ref{tab:mixed_k_2}, \ref{tab:mixed_k_3} the discretization errors for $k=1,2,3$ are shown. 
The first column shows the refinement level $L$, i.e.~the number of uniform $h$-refinements of $\Omega$. The column "order" contains the error reduction relative to the previous level.
The experiments show optimal convergence rates for $\primal$ and $\aux$ as predicted by the analysis. 
For the columns containing the errors for $\vp$ and $\vphi$, the (analytically not available) exact solutions $\vp$ and $\vphi$ are replaced by their numerical solutions on level $L = 9$. Note that also for $\vp$ and $\vphi$ optimal convergence rates are observed.

\begin{table}[H]
\caption{Discretization errors, $k=1$. \label{tab:mixed_k_1}}
\centering
\resizebox{\textwidth}{!}{
\begin{tabular}{c|cc|cc|cc|cc}
  \hspace*{\fill} \rule[-1ex]{0ex}{3.5ex}$L$ \hspace*{\fill}
    & \makebox[10ex][c]{$\|\primal-\primal_h\|_1$} 
    & \makebox[4ex][c]{order} 
    & \makebox[10ex][c]{$\|\aux-\aux_{h}\|_0$} 
    & \makebox[4ex][c]{order} 
    & \makebox[10ex][c]{$\|\vp-\vp_{h}\|_0$} 
    & \makebox[4ex][c]{order} 
    & \makebox[10ex][c]{$\|\vphi-\vphi_{h}\|_1$} 
    & \makebox[4ex][c]{order}\\
  \hline
    4 &  $1.09 \cdot 10^{-1}$ & $0.992$ & $1.24 \cdot 10^{-1}$ & $0.974$ & $1.43 \cdot 10^{-2}$ & $1.985$ & $1.04 \cdot 10^{-1}$ & $1.050$ \rule[-1ex]{0ex}{3.5ex}\\ 
    5 &  $5.47 \cdot 10^{-2}$ & $0.998$ & $6.26 \cdot 10^{-2}$ & $0.993$ & $3.59 \cdot 10^{-3}$ & $1.999$ & $5.17 \cdot 10^{-2}$ & $1.017$\\
    6 &  $2.73 \cdot 10^{-2}$ & $0.999$ & $3.13 \cdot 10^{-2}$ & $0.998$ & $8.90 \cdot 10^{-4}$ & $2.011$ & $2.56 \cdot 10^{-2}$ & $1.012$\\
    7 &  $1.36 \cdot 10^{-2}$ & $0.999$ & $1.56 \cdot 10^{-2}$ & $0.999$ & $2.14 \cdot 10^{-4}$ & $2.052$ & $1.24 \cdot 10^{-2}$ & $1.036$\\
\end{tabular}
}
\end{table}

\begin{table}[H]
\caption{Discretization errors, $k=2$. \label{tab:mixed_k_2}}
\centering
\resizebox{\textwidth}{!}{
\begin{tabular}{c|cc|cc|cc|cc}
  \hspace*{\fill} \rule[-1ex]{0ex}{3.5ex}$L$ \hspace*{\fill}
    & \makebox[10ex][c]{$\|\primal-\primal_h\|_1$} 
    & \makebox[4ex][c]{order} 
    & \makebox[10ex][c]{$\|\aux-\aux_{h}\|_0$} 
    & \makebox[4ex][c]{order} 
    & \makebox[10ex][c]{$\|\vp-\vp_{h}\|_0$} 
    & \makebox[4ex][c]{order} 
    & \makebox[10ex][c]{$\|\vphi-\vphi_{h}\|_1$} 
    & \makebox[4ex][c]{order}\\
  \hline
    4 &  $4.33 \cdot 10^{-2}$ & $2.071$ & $1.75 \cdot 10^{-1}$ & $2.104$ & $1.22 \cdot 10^{-2}$ & $3.160$ & $2.04 \cdot 10^{-1}$ & $2.055$ \rule[-1ex]{0ex}{3.5ex}\\ 
    5 &  $1.06 \cdot 10^{-2}$ & $2.018$ & $4.29 \cdot 10^{-2}$ & $2.030$ & $1.49 \cdot 10^{-3}$ & $3.042$ & $5.05 \cdot 10^{-2}$ & $2.017$\\
    6 &  $2.66 \cdot 10^{-3}$ & $2.004$ & $1.06 \cdot 10^{-2}$ & $2.008$ & $1.85 \cdot 10^{-4}$ & $3.010$ & $1.25 \cdot 10^{-2}$ & $2.005$\\
    7 &  $6.65 \cdot 10^{-4}$ & $2.001$ & $2.66 \cdot 10^{-3}$ & $2.002$ & $2.30 \cdot 10^{-5}$ & $3.002$ & $3.13 \cdot 10^{-3}$ & $2.003$\\
\end{tabular}
}
\end{table}

\begin{table}[H]
\caption{Discretization errors, $k=3$. \label{tab:mixed_k_3}}
\centering
\resizebox{\textwidth}{!}{
\begin{tabular}{c|cc|cc|cc|cc}
  \hspace*{\fill} \rule[-1ex]{0ex}{3.5ex}$L$ \hspace*{\fill}
    & \makebox[10ex][c]{$\|\primal-\primal_h\|_1$} 
    & \makebox[4ex][c]{order} 
    & \makebox[10ex][c]{$\|\aux-\aux_{h}\|_0$} 
    & \makebox[4ex][c]{order} 
    & \makebox[10ex][c]{$\|\vp-\vp_{h}\|_0$} 
    & \makebox[4ex][c]{order} 
    & \makebox[10ex][c]{$\|\vphi-\vphi_{h}\|_1$} 
    & \makebox[4ex][c]{order}\\
  \hline
    4 &  $2.75 \cdot 10^{-3}$ & $3.084$ & $1.10 \cdot 10^{-2}$ & $3.105$ & $7.69 \cdot 10^{-4}$ & $4.219$ & $1.27 \cdot 10^{-2}$ & $3.019$ \rule[-1ex]{0ex}{3.5ex}\\ 
    5 &  $3.46 \cdot 10^{-4}$ & $2.994$ & $1.38 \cdot 10^{-3}$ & $2.989$ & $4.63 \cdot 10^{-5}$ & $4.054$ & $1.62 \cdot 10^{-3}$ & $2.966$\\
    6 &  $4.37 \cdot 10^{-5}$ & $2.985$ & $1.75 \cdot 10^{-4}$ & $2.978$ & $2.86 \cdot 10^{-6}$ & $4.012$ & $2.07 \cdot 10^{-4}$ & $2.972$\\
    7 &  $5.50 \cdot 10^{-6}$ & $2.989$ & $2.22 \cdot 10^{-5}$ & $2.985$ & $1.78 \cdot 10^{-7}$ & $4.005$ & $2.60 \cdot 10^{-5}$ & $2.995$\\
\end{tabular}
}
\end{table}

\bibliographystyle{abbrv}
\bibliography{KLplate_bibliography}

\end{document}